\documentclass[11pt]{amsart}
\usepackage{amsmath}
\usepackage{amssymb, verbatim}
\usepackage{pstricks,pst-node,pst-plot}
\usepackage{comment}
\usepackage{color}
\usepackage{extarrows}

\title[]{The Dirichlet Problem for the $k$-Hessian Equation on a complex manifold}
\author[T. C. Collins]{Tristan C. Collins}
  \email{tristanc@mit.edu}
  \address{Department of Mathematics, Massachusetts Institute of Technology, 77 Massachusetts Avenue, Cambridge, MA 02139}
 \thanks{T.C.C is supported in part by NSF grant DMS-1810924 and an Alfred P. Sloan Fellowship. }

 \author[S. Picard]{Sebastien Picard}
  \email{spicard@math.harvard.edu}
  \address{Department of Mathematics, Harvard University, 1 Oxford St.,Cambridge, MA, 02138}
  \thanks{}  

\theoremstyle{plain}
\newtheorem{thm}{Theorem}[section]
\newtheorem{prop}[thm]{Proposition}

\newtheorem{lem}[thm]{Lemma}

\theoremstyle{definition}

\numberwithin{equation}{section}

\newcommand{\del}{\partial}

\newcommand{\dbar}{\overline{\del}}
\newcommand{\ddb}{\sqrt{-1}\del\dbar}

\newcommand{\F}{\mathcal{F}}

\newcommand{\p}{\partial}
\newcommand{\be}{\begin{equation}}
\newcommand{\bea}{\begin{eqnarray}}
\newcommand{\eea}{\end{eqnarray}} 
\newcommand{\ee}{\end{equation}}

\renewcommand{\leq}{\leqslant}
\renewcommand{\geq}{\geqslant}

\renewcommand{\epsilon}{\varepsilon}
\renewcommand{\phi}{\varphi}

\begin{document}

\maketitle

\begin{abstract}
We solve the Dirichlet problem for $k$-Hessian equations on compact complex manifolds with boundary, given the existence of a subsolution. Our method is based on a second order a priori estimate of the solution on the boundary with a particular gradient scale. The scale allows us to apply a blow-up argument to obtain control on all necessary norms of the solution.
  
\end{abstract}

\section{Introduction}
\par Nonlinear partial differential equations involving elementary symmetric polynomials appear throughout differential geometry. We start by describing the model setup, and then specialize to the setting of the current paper and discuss how the $\sigma_k$ operator arises in complex differential geometry.
\smallskip
\par Let $\lambda = (\lambda_1, \dots, \lambda_n) \in \mathbb{R}^n$. For $k \in \{1, \dots, n\}$, we will denote the $k$-th elementary symmetric polynomial by
\[
\sigma_k(\lambda) = \sum_{1 \leq j_1 < j_2 < \cdots < j_k \leq n} \lambda_{j_1} \cdots \lambda_{j_n}. 
\]
As it will be convenient later, we will also sometimes use the convention $\sigma_0(\lambda)=1$ and $\sigma_\ell(\lambda)=0$ for $\ell > n$.
\smallskip
\par Let $\Omega \subset \mathbb{R}^n$ be a bounded domain with smooth boundary. Let $\psi : \Omega \rightarrow \mathbb{R}$ with $\psi>0$ and $\phi: \partial \Omega \rightarrow \mathbb{R}$ be given smooth functions. The Dirichlet problem for the $\sigma_k$ operator seeks a function $u: \Omega \rightarrow \mathbb{R}$ solving
\begin{equation} \label{model-sigma-k}
\begin{aligned}
\sigma_k(\lambda) &= \psi(x),\\
u|_{\del \Omega} &= \phi.
\end{aligned}
\end{equation}
where $\lambda=(\lambda_1,\dots,\lambda_n)$ are the eigenvalues of $D^2 u$.  As for many nonlinear equations, we must restrict our search to admissible functions $u$ in order to ensure the ellipticity of~\eqref{model-sigma-k}. In this case, the admissibility condition requires that $\lambda(D^2u) \in \Gamma_k$, where the set $\Gamma_k \subset \mathbb{R}^n$ is defined by
\[
\Gamma_k = \{ \lambda \in \mathbb{R}^n : \sigma_1(\lambda)>0, \dots, \sigma_k(\lambda)> 0\}.
\]
It is a result of G\aa rding \cite{Garding} that $\Gamma_k$ is a convex cone. After early work by \cite{CNS1,ChengYau77,Ivo,Ivo83,Kr83}, this problem was solved for a unique admissible solution by Caffarelli-Nirenberg-Spruck \cite{CNS3} under the condition that the boundary $\partial \Omega$ is $(k-1)$-convex. The proof of Caffarelli-Nirenberg-Spruck \cite{CNS3} was subsequently simplified by Trudinger \cite{T}.
\smallskip
\par In \cite{Guan94,Guan98}, B. Guan replaced the condition of $(k-1)$-convexity of the boundary by the condition that the domain should admit a subsolution $\underline{u}$. This argument no longer relies on the shape of the boundary, and has found applications in geometric problems e.g. \cite{CY,GuanSpruck93,GuanPF,GuanZh,PS09}. We will adopt the subsolution approach rather than impose a condition on the boundary of our space. We note that the Dirichlet problem for an arbitrary domain may not admit an admissible solution without such a condition, as can be seen in the simplest case of $k=n$ with constant boundary data, which forces the domain to be convex.
\smallskip
\par The Dirichlet problem (\ref{model-sigma-k}) can also be studied when $\lambda$ is the vector of eigenvalues of the complex Hessian $\ddb u$. The Dirichlet problem for complex $k$-Hessian equations in a domain $\Omega \subset \mathbb{C}^n$ was solved by Vinacua \cite{Vin} and Li \cite{LiSY} building on earlier work of \cite{CKNS,ChengYau80,Guan98}. This paper concerns the global version of this result on complex manifolds.
\smallskip
\par In complex geometry, the problem takes the following form. Let $(X,\alpha)$ be a compact complex manifold with Hermitian metric $\alpha$. Let $[\chi] \in H^{1,1}_{BC}(X,\mathbb{R})$ be a given Bott-Chern cohomology class. Recall that
\[
H_{BC}^{1,1}(X,\mathbb{R}) = { \{ \alpha \in \Omega^{1,1}(X,\mathbb{R}) \} \over \{ \ddb f : f \in C^\infty(X,\mathbb{R}) \}} .
\]
Suppose $[\chi]$ admits a representative, denoted $\chi$, which is $k$-positive, meaning that the eigenvalues of $\chi$ with respect to $\alpha$ lie in the cone $\Gamma_k$. In this case, we say that $[\chi]$ is a $k$-positive class. For example, when $k=n$, an $n$-positive class is a K\"ahler class.  We now ask: given $\psi$ a positive smooth function, does there exist a representative $\chi' \in [\chi]$ with prescribed measure
\[
\chi'^k \wedge \alpha^{n-k} = \psi\alpha^n.
\]
This is a nonlinear equation for a potential function $u$, which can be written as
\begin{equation} \label{k-hessian}
(\chi+ \ddb u)^k \wedge \alpha^{n-k} = \psi \, \alpha^n.
\end{equation}
On a complex manifold with non-empty boundary $\partial X$, we must also prescribe boundary data
\[
u|_{\partial X} = \phi,
\]
where $\phi \in C^\infty(\partial X, \mathbb{R})$ is a given function.
\smallskip
\par When $k=n$, equation (\ref{k-hessian}) is the complex Monge-Amp\`ere equation, which was solved by S.-T. Yau \cite{Yau} on closed K\"ahler manifolds in the resolution of the Calabi conjecture. The analogous problem for the complex Monge-Amp\`ere equation on closed Hermitian manifolds was solved by Tosatti-Weinkove \cite{TW10a,TW10b}. On the other hand, the Dirichlet problem for the complex Monge-Amp\`ere equation on manifolds with boundary was studied by Cherrier-Hanani \cite{CH} on strongly pseudoconvex manifolds, and solved by Guan-Li \cite{GuanLi10} on manifolds admitting a subsolution. For an overview of the vast field of complex Monge-Amp\`ere equations, we refer the reader to the survey of Phong-Song-Sturm \cite{PSS}.
\smallskip
\par When $k=1$, the equation becomes a linear PDE whose solvability is well-known. For $1<k<n$, equations of this type were discovered by Fu-Yau in connection to the Hull-Strominger system of heterotic string compactifications \cite{FY}. See e.g. \cite{CHZ,FGV,PPZ2,PPZ3,PPZ4,PPZ5} for the study of $k$-Hessian equations in the context of Fu-Yau compactifications.
\smallskip
\par In \cite{Blocki05}, B\l ocki developed a pluripotential theory for complex $k$-Hessian equations in a domain in $\mathbb{C}^n$. Since then, the theory of weak solutions of complex Hessian equations has been extended to closed complex manifolds \cite{DK14}. There has been much recent work on weak solutions to complex $k$-Hessian equations, in particular aimed at develoing a potential theory parallel to the Bedford-Taylor \cite{BT} theory of complex Monge-Amp\`ere equations; see e.g. \cite{DK14,DC15,Lu13b,LN15,HL,KN,Ngu} and the references therein.
\smallskip
\par Concerning strong solutions of (\ref{k-hessian}), after initial progress by Hou \cite{Hou}, Kokarev \cite{Kok} and Jbilou \cite{Jbi}, the problem on K\"ahler manifolds without boundary was solved in 2012 by combining the Liouville theorem of Dinew-Ko\l odziej \cite{DK17} with the second order estimate of Hou-Ma-Wu \cite{HMW}. The corresponding problem on Hermitian manifolds was solved by D.K. Zhang \cite{ZhangDK} and Sz\'ekelyhidi \cite{Sz}. The second order estimate has since been refined and extended in various directions, see e.g. \cite{DPZ,DL,PhongTo,PPZ1}.
\smallskip
\par After the solution of Dinew-Ko\l odziej/Hou-Ma-Wu on closed K\"ahler manifolds, the remaining problem was to solve the Dirichlet problem for equation (\ref{k-hessian}) on complex manifolds with boundary. Gu and Nguyen \cite{GuNgu} were able to obtain smooth solutions on a small ball, and use these solutions together with a balayage argument to obtain continuous solutions when $\alpha$ is locally conformally K\"ahler. Feng-Ge-Zheng \cite{FGZ} considered a more general class of equations and reduced the problem to obtaining an a priori estimate on the gradient of the solution. However, obtaining an a priori gradient estimate for complex $k-$Hessian equations on manifolds via the maximum principle is a known open problem in the field. To our knowledge, the only result so far when $1<k<n$ is the work of X.-W. Zhang \cite{ZhangXW} where the gradient estimate is obtained under the assumption that $\chi + \ddb u >0$.
\smallskip
\par In this paper, we solve the Dirichlet problem for (\ref{k-hessian}). Rather than use the maximum principle for the gradient estimate, we use a blow-up argument and apply the Liouville theorem of Dinew-Ko\l odziej \cite{DK17}. For this to work, the second order estimate needs to scale correctly. Namely, we require
\[
\sup_X \| \ddb u \|_{(X,\alpha)} \leq C(1+ \sup_X \| \nabla u \|^2_{(X,\alpha)}),
\]
where the constant $C$ only depends on the background data and is independent of the solution $u$. Thus the main difficulty is shifted from obtaining a gradient estimate to obtaining a second order estimate with a particular scale. Our main contribution is to obtain this estimate on the boundary $\p X$, which allows us to solve the Dirichlet problem.  A similar argument, restricted to the setting of the complex Monge-Amp\`ere equation, appears in \cite{Bouck}.  Our main theorem is

\begin{thm} \label{thm-main}
Let $(X,\alpha)$ be a compact Hermitian manifold with boundary. Let $\chi \in \Gamma_k(X,\alpha)$ be a $(1,1)$ form, $\psi \in C^\infty(X)$ a smooth function satisfying $\psi \geq c >0$, and $\phi \in C^\infty(\p X, \mathbb{R})$. Suppose there exists a subsolution $\underline{u} \in C^\infty(\overline{X},\mathbb{R})$ satisfying
\[
\sigma_k(\underline{\lambda}) \geq \psi, \ \ \underline{u}|_{\p X} = \phi
\]
where $\underline{\lambda} \in \Gamma_k$ are the eigenvalues of $\chi + \ddb \underline{u}$ with respect to $\alpha$. Then there exists a unique $u \in C^\infty(\overline{X}, \mathbb{R})$ solving the equation
\[
\sigma_{k}(\lambda) = \psi, \ \ u|_{\p X} = \phi,
\]
with $\lambda \in \Gamma_k$, where $\lambda$ denotes the eigenvalues of $\chi + \ddb u$ with respect to $\alpha$. 
\end{thm}

The paper is organized as follows. In \S 2, we establish notation and use the continuity method to reduce Theorem~\ref{thm-main} to obtaining a priori estimates on the solution. In \S 3, we recall various estimates which are known in the literature and will be used in the proof. In \S 4-5, we prove the boundary $C^2$ estimate. The most intricate part of the argument is the double normal estimate in \S 5, and here we build on the technique of Caffarelli-Nirenberg-Spruck \cite{CNS3}. Finally, in \S 6 we combine the Liouville theorem of Dinew-Ko\l odziej \cite{DK17} with a blow-up argument to complete the proof.
\bigskip
\par {\bf Acknowledgements:} The authors are grateful to D. H. Phong and X. Zhang for helpful comments and suggestions.

\section{Setup}

\subsection{Notation}
Let $(X,\alpha)$ be a compact complex manifold with Hermitian metric $\alpha$ and non-empty boundary $\p X$. In local coordinates, we write
\[
\alpha = \sqrt{-1} \alpha_{\bar{k} j} dz^j \wedge d \bar{z}^k
\]
and $\alpha^{j \bar{k}}$ for the inverse of $\alpha_{\bar{k} j}$, so that $\alpha^{i \bar{k}} \alpha_{\bar{k} j} = \delta^i{}_j$. Covariant derivatives $\nabla$ will be with respect to the Chern connection of $\alpha$, which acts on sections $W \in \Omega^{1,0}(X)$ by
\be \label{covariant-derivatives}
\nabla_i W_k = \partial_i W_k - \Gamma^r{}_{ik} W_r, \ \ \nabla_{\bar{i}} W_k = \partial_{\bar{i}} W_k,
\ee
with $\Gamma^r{}_{ik} = g^{r \bar{\ell}} \partial_i g_{\bar{\ell} k}$. 
\smallskip
\par Let $\chi \in \Omega^{1,1}(X,\mathbb{R})$ be a differential form of type $(1,1)$, written in local coordinates as $\chi = \sqrt{-1} \chi_{\bar{k} j} dz^j \wedge d \bar{z}^k$. We say
\[
\chi \in \Gamma_k(X,\alpha)
\]
if the vector of eigenvalues of the hermitian endomorphism $\alpha^{i \bar{k}} \chi_{\bar{k} j}$ lies in the $\Gamma_k$ cone at each point.
\smallskip
\par Let $\chi \in \Gamma_k(X,\alpha)$, let $\psi\geq c >0$ be a smooth positive function on $X$, and let $\phi \in C^\infty(\p X, \mathbb{R})$. We seek a potential function $u \in C^\infty(X,\mathbb{R})$ solving the equation
\[
\sigma_k(\lambda) = \psi, \ \ u|_{\p X} = \phi,
\]
where $\lambda = (\lambda_1, \dots, \lambda_n) \in \Gamma_k$ are the eigenvalues of the endomorphism
\[
h^i{}_j = \alpha^{i \bar{k}}(\chi_{\bar{k} j} + u_{\bar{k} j}).
\]
When forming a vector out of eigenvalues of an endomorphism, we will use the ordering
\[
\lambda_n \leq \lambda_{n-1} \leq \cdots \leq \lambda_1.
\]
A subsolution $\underline{u}$ to our Dirichlet problem is a smooth function satisfying
\[
\sigma_k(\underline{\lambda}) \geq \psi, \ \ \underline{u}|_{\p X} = \phi,
\]
where $\underline{\lambda} \in \Gamma_k$ are the eigenvalues of $\alpha^{i \bar{k}} (\chi_{\bar{k} j} + \underline{u}_{\bar{k} j})$.
\smallskip
\par We define the tensor
\[
\sigma_k^{p \bar{q}} =  {\partial \sigma_k \over \partial h^r{}_p} \alpha^{r \bar{q}},
\]
and also use the notation
\[
\F = \sigma_k^{p \bar{q}} \alpha_{\bar{q} p}.
\]
Here the notation ${\partial \sigma_k \over \partial h^r{}_p}$ means the derivative of $\sigma_k$ regarded as a function on hermitian matrices.  At a diagonal matrix $h^i{}_j$, we have the formula \cite{Baller}
\be \label{sigma-k-ij}
{\partial \sigma_k \over \partial h^i{}_j} = \delta_{ij} {\partial \sigma_k \over \partial \lambda_i}.
\ee
Therefore, at a point $p \in X$ where $\alpha_{\bar{k} j} = \delta_{kj}$ and $h^i{}_j = \lambda_j \delta^i{}_j$, then
\[
\sigma_k^{p \bar{q}} = \sigma_{k-1}(\lambda|p) \delta_{pq}, \ \ \ \F = \sum_p \sigma_{k-1}(\lambda|p).
\]
Here we use the notation $(\lambda|i) \in \mathbb{R}^n$ for the vector where the $i$-th component of $\lambda$ has been replaced by 0, which allows us to write
\[
{\partial \sigma_k \over \partial \lambda_i} = \sigma_{k-1}(\lambda|i).
\]
It will be convenient to denote
\[
K = 1 + \| \nabla u \|^2_{L^\infty(X,\alpha)}.
\]
Finally, we note that we will use the usual convention where $C$ denotes a constant which may change line by line, but is only allowed to depend on $(X,\alpha)$, $\chi$, $\psi$, $\phi$, $\underline{u}$.

\subsection{Continuity Method}
Our goal is to prove the following a priori estimate. This will allow us to use the continuity method to solve the Dirichlet problem.

\begin{thm} \label{thm-c2-uniform}
Let $(X,\alpha)$ be a compact Hermitian manifold with boundary. Let $\chi \in \Gamma_k(X,\alpha)$ be a $(1,1)$ form, $\psi \in C^\infty(X)$ a smooth function satisfying $\psi \geq c >0$, and $\phi \in C^\infty(\p X, \mathbb{R})$. Suppose $u \in C^4(\overline{X}, \mathbb{R})$ solves the equation
\[
\sigma_{k}(\lambda) = \psi, \ \ u|_{\p X} = \phi,
\]
where $\lambda \in \Gamma_k$ are the eigenvalues of $\chi + \ddb u$ with respect to $\alpha$. Suppose there exists a subsolution $\underline{u} \in C^\infty(\overline{X},\mathbb{R})$ satisfying
\[
\sigma_k(\underline{\lambda}) \geq \psi, \ \ \underline{u}|_{\p X} = \phi
\]
where $\underline{\lambda} \in \Gamma_k$ are the eigenvalues of $\chi + \ddb \underline{u}$ with respect to $\alpha$. Then
\[
\| u \|_{L^\infty(X)} + \| \nabla u \|_{L^\infty(X,\alpha)} + \| \ddb u \|_{L^\infty(X,\alpha)} \leq C,
\]
where $C$ depends on $(X,\alpha)$, $\chi$, $\underline{u}$, and $\| \psi \|_{C^2}$ and $\inf_{X}\psi$.
\end{thm}

We now give the standard argument which shows that this a priori estimates implies the main theorem (Theorem \ref{thm-main}).
\smallskip
\par For a parameter $t \in [0,1]$, we consider the family
\be \label{cont-meth}
\sigma_k(\lambda_t) = \psi_{t} := t \psi + (1-t) \sigma_k(\underline{\lambda}),  \ \ u_t|_{\p X} = \phi,
\ee
where $\lambda_t \in \Gamma_k$ are the eigenvalues of $\alpha^{j \bar{k}} (\chi_{\bar{k} j} + (u_t)_{\bar{k} j})$. Let $\alpha \in (0,1)$ to be determined later, and define 
\[
S = \{ t \in [0,1] : {\rm there} \ {\rm exists} \ u_t \in C^{4,\alpha}(\overline{X},\mathbb{R}), \ {\rm with} \ \lambda_t \in \Gamma_k, \ {\rm solving} \ (\ref{cont-meth}) \}.
\]
The function $u_0= \underline{u}$ solves the equation at $t=0$, hence $S$ is non-empty. The linearization of the operator $u \mapsto \sigma_k(\lambda)$ at $\lambda \in \Gamma_k$ is the complex Laplacian using the Hermitian metric whose inverse is $\sigma_k^{p \bar{q}}(\lambda)$. Since this operator is invertible, we have that $S$ is an open set by the implicit function theorem.
\smallskip
\par To show $S$ is closed, we use the a priori estimates. Let $t_i \in S$ be a sequence converging to some $t_\infty \in [0,1]$. We have $\psi_t \geq \psi = \psi_1 \geq a >0$ and so
\[
\sigma_k(\underline{\lambda}) \geq t \psi + (1-t) \sigma_k(\underline{\lambda}) \geq a >0,
\]
hence $\underline{u}$ is a subsolution along the continuity path. By Theorem \ref{thm-c2-uniform}, we have that
\[
\| u_{t_i} \|_{L^\infty(X)} + \| \nabla u_{t_i} \|_{L^\infty(X,\alpha)} + \| \ddb u_{t_i} \|_{L^\infty(X,\alpha)} \leq C
\]
uniformly along the path.  This estimate implies that we have uniform ellipticity
\[
C \geq \sigma_{k-1}(\lambda_t|i) \geq C^{-1},
\]
of the linearized $\sigma_k$ operator along the path. The uniform ellipticity can be seen by the inequality (\ref{sigma11-lower-bdd}), which will be discussed later.
\smallskip
\par By the envelope trick of Y. Wang \cite{WangYu}, as generalized in \cite{TWWY} (see also \cite{CJY}), we can extend $\sigma_k^{\frac{1}{k}}$ to a concave operator on the set of real symmetric matrices, which remains uniformly elliptic along the continuity path. To obtain $C^{2,\alpha}$ estimates, we can now invoke the Evans-Krylov theorem \cite{E, Kr1, Kr} in the interior of $X$, and the Krylov theorem \cite{Kr} near the boundary $\p X$ (see \cite[Theorem 7.1]{ChenWu}, or alternatively \cite{SS}).  Therefore
\[
\| u_{t_i} \|_{C^{2,\alpha}(X)} \leq C.
\]
Differentiating the equation now gives a uniformly elliptic PDE with uniformly H\"older coefficients. Applying the Schauder estimates gives
\[
\| u_{t_i} \|_{C^{4,\alpha}(X)} \leq C.
\]
We can now take convergent subsequence to a limiting function $u_\infty \in C^{4,\alpha}$ which solves the equation at $t_\infty$. Since $S$ is nonempty, open and closed, $S=[0,1]$.
\bigskip
\par This gives the existence of a $C^{4,\alpha}$ solution to the Dirichlet problem. By differentiating the equation and invoking Schauder theory, we see that this solution is in fact smooth. Uniqueness follows from the maximum principle. This proves Theorem \ref{thm-main} given Theorem \ref{thm-c2-uniform}.

\section{Preliminary Estimates}
In this section, we suppose $u$ satisfies $\sigma_k(\lambda)=\psi \geq c>0$ with $\lambda \in \Gamma_k$ and $u|_{\p X} = \phi$, and recall several estimates which are well-known in the literature. We start with the maximum principle.

\begin{lem} \label{lem-max-princ} \cite{CNS3} Let $X$ be a compact complex manifold with boundary. Suppose $v : X \rightarrow \mathbb{R}$ is a smooth function such that the vector of eigenvalues of $\alpha^{-1}(\chi+ \ddb v)$ lies outside the set $\{ \lambda \in \Gamma_k : \sigma_k(\lambda) \geq \psi \}$ at all points $x \in X$, and $u \leq v$ on $\p X$. Then
\[ 
u \leq v
\]
on $X$.
\end{lem}

\par \noindent {\it Proof:} Suppose $u>v$ at some point in $X$, so that $u-v$ attains a maximum at a point $p$ in the interior of $X$. Then $D^2 u (p) \leq D^2 v(p)$. Choose normal coordinates for $\alpha$, and let $\lambda(p) = (\lambda_1,\dots,\lambda_n)$ be the eigenvalues of $\alpha^{-1}(\chi+ \ddb u)$, arranged in decreasing, and let $\mu(p) = (\mu_1,\dots,\mu_n)$ be the eigenvalues of $\alpha^{-1}(\chi+ \ddb v)$, arranged in decreasing order. By the Weyl inequality 
\[
\lambda_i \leq \mu_i, \ \ i \in \{1, \dots, n\}.
\]
Since $\lambda(p)$ is in $\{ \lambda \in \Gamma_k : \sigma_k(\lambda) \geq \psi \}$, so is $\mu(p)$. This is a contradiction. $\qed$

\bigskip
\par Let $\underline{u}$ be a subsolution, which satisfies  $\sigma_k(\underline{\lambda}) \geq \psi$ and $\underline{u}|_{\p X} = \phi$. Let $b(z)$ be a function satisfying
\[
\alpha^{j \bar{k}} (\chi_{\bar{k} j} + \p_j \p_{\bar{k}} b) = 0, \ \ \ b|_{\p X} = \phi.
\]
Such a function can be constructed by solving the linear equation with homogeneous boundary condition
\[
\alpha^{j \bar{k}} \p_j \p_{\bar{k}} \tilde{b} = - \alpha^{j \bar{k}} (\chi_{\bar{k} j} + \underline{u}_{\bar{k} j}), \ \ \ \tilde{b}|_{\partial X} = 0
\]
and then letting $b = \tilde{b} + \underline{u}$.
\smallskip
\par Note that $\sigma_1(\lambda)>0$ since $\lambda \in \Gamma_k$. By the comparison principle,
\be \label{barriers}
\underline{u} \leq u \leq b,
\ee
and $\underline{u}=u=b$ on $\p X$.  We obtain
\begin{lem}\label{lem:c0-c1-bdry}
In the setting of Theorem~\ref{thm-c2-uniform} we have the following estimates
\be
\sup_X \| u \|_{L^\infty} \leq C,
\ee
\be \label{bdd-grad}
\sup_{\p X} \| \nabla u \|_{(\p X, \alpha|_{\p X})} \leq C.
\ee
where $C$ depends on $(X,\alpha)$, $\chi$, $\underline{u}$.
\end{lem}
Indeed, to see the boundary gradient estimate, fix a point $p \in \partial X$, and let $t^\alpha$ denote coordinates tangential to $\del X$ at $p$, and $x^n$ a coordinate parallel the inner normal direction. We have $\partial_{t^\alpha} (u-\underline{u}) = 0$, $\partial_{x^n} (u-\underline{u}) \geq 0$, and $\partial_{x^n} (b-u) \geq 0$ at the point $p$, which implies the boundary gradient estimate.
\smallskip
\par It was first proved by Hou-Ma-Wu \cite{HMW} that complex Hessian equations on closed complex manifolds satisfy an a priori estimate on the complex Hessian $ \ddb u$. This was generalized to the Hermitian case by \cite{Sz,ZhangDK}. Sz\'ekelyhidi \cite{Sz} further  generalized this estimate to a wide class of nonlinear equations on closed Hermitian manifolds. Using the notation of \cite{Sz}, consider the test function
\[
G = \log \lambda_1 + \phi(|\nabla u|^2) + \psi(u),
\]
with
\[
\phi(t) = - {1 \over 2} \log (1 - t/2K), \ \ \psi(t) = -2At + {A \tau \over 2} t^2,
\]
for constants $A,\tau$ depending on $(X,\alpha)$, $\chi$, $\psi$. If $G$ attains a maximum in the interior of $X$, we may apply the maximum principle argument of (Proposition 13, \cite{Sz}) and obtain the estimate
\[
\sup_X \| \ddb u \|_{(X,\alpha)} \leq CK,
\]
where $C$ depends on $(X,\alpha), \underline{u}, \phi$, $|\psi^{\frac{1}{k}}|_{C^{2}(X)}$, and $\chi$. Next, we consider the case when $G$ attains a maximum on the boundary of $X$.
\[
\sup_X \left[\log \lambda_1 + \phi(|\nabla u|^2) + \psi(u) \right] \leq \sup_{\p X} \left[ \log \lambda_1 + \phi(|\nabla u|^2) + \psi(u) \right].
\]
In that case, $\| \ddb u \|_{(X,\alpha)}$ can be estimated by its supremum on the boundary. We can thus conclude from the Hou-Ma-Wu maximum principle

\begin{prop} \label{prop-c2-est-HMW}
Let $(X,\alpha)$ be a compact Hermitian manifold with boundary. Let $\chi \in \Gamma_k(X,\alpha)$ be a $(1,1)$ form, $\psi \in C^\infty(X)$ a smooth function satisfying $\psi >0$, and $\phi \in C^\infty(\p X, \mathbb{R})$. Suppose $u \in C^4(\overline{X}, \mathbb{R})$ solves the equation
\[
\sigma_{k}(\lambda) = \psi, \ \ u|_{\p X} = \phi,
\]
where $\lambda \in \Gamma_k$ are the eigenvalues of $\chi + \ddb u$ with respect to $\alpha$. Suppose there exists a subsolution $\underline{u} \in C^\infty(\overline{X},\mathbb{R})$ satisfying
\[
\sigma_k(\underline{\lambda}) \geq \psi, \ \ \underline{u}|_{\p X} = \phi
\]
where $\underline{\lambda} \in \Gamma_k$ are the eigenvalues of $\chi + \ddb \underline{u}$ with respect to $\alpha$. Then
\be\label{C2-interior}
 \| \ddb u \|_{L^\infty(X,\alpha)} \leq C(1+ \sup_{\del X} |\ddb u| + \sup_{X}|\nabla u|^2),
\ee
where $C$ depends on $(X,\alpha)$, $\underline{u}, \phi$,  $|\psi^{\frac{1}{k}}|_{C^{2}(X)}$, and $\chi$.
\end{prop}


Our main task is to estimate $\| \ddb u \|$ on the boundary $\p X$, which will be done in the following sections.
\smallskip
\par Before proceeding, we note a few more inequalities of elementary symmetric polynomials which will be used. For $\lambda$ in the cone $\Gamma_k$, we have the generalized Newton-Maclaurin inequalities
\be \label{newton-mac}
\left( {H_k \over H_s} \right)^{1/(k-s)} \leq \left( {H_\ell \over H_r} \right)^{1/(\ell-r)}, \ \ H_k = {\sigma_k(\lambda) \over {n \choose k}},
\ee
provided $k > s \geq 0$, $\ell > r \geq 0$, and $k \geq \ell$, $s \geq r$ (see e.g. \cite{Spruck}). 
\smallskip
\par Let $\lambda = (\lambda_1, \dots, \lambda_n) \in \Gamma_k$ be ordered such that $\lambda_1 \geq \dots \geq \lambda_n$. Then $\lambda_1 > 0$. As noted in Hou-Ma-Wu \cite{HMW}, we can estimate
\be \label{sigma11-lower-bdd}
\sigma_{k-1}(\lambda|1) \geq {k \over n} {\sigma_k(\lambda) \over \lambda_1}.
\ee
Indeed, following \cite{HMW}, rearranging the identity $\sigma_k = \lambda_1 \sigma_{k-1}(\lambda|1) + \sigma_k(\lambda|1)$ gives
\be \label{sigma-k-identity}
\lambda_1 \sigma_{k-1}(\lambda|1) = \sigma_k - \sigma_k(\lambda|1).
\ee
By the generalized Newton-Maclaurin inequalities
\[
{\sigma_k(\lambda|1) \over \sigma_1(\lambda|1)} \leq {n-k \over k(n-1)} \sigma_{k-1}(\lambda|1).
\]
Since $\sigma_1(\lambda|1) \leq (n-1) \lambda_1$, we obtain
\[
\sigma_k(\lambda|1) \leq \left({n \over k}-1\right) \lambda_1 \sigma_{k-1}(\lambda|1),
\]
and substituting this into (\ref{sigma-k-identity}) gives (\ref{sigma11-lower-bdd}). 
\smallskip
\par As a consequence of (\ref{sigma11-lower-bdd}), for every index $i \in \{1, \dots, n\}$,
\be \label{ellipticity}
\sigma_{k-1}(\lambda|i) > 0,
\ee
since $\sigma_{k-1}(\lambda|i) \geq \sigma_{k-1}(\lambda|1)$.
\smallskip
\par Next, as stated in \cite{Guan14}, we have the following 
\begin{lem}
Suppose $\lambda \in \Gamma_k$.  Then, for any $\lambda \in \Gamma_k$, index $r$, and $\epsilon>0$, we may estimate
\be \label{quad-split}
\sum_i \sigma_{k-1}(\lambda|i) |\lambda_i| \leq \epsilon \sum_{i \neq r} \sigma_{k-1}(\lambda|i) \lambda_i^2 + {C \over \epsilon} \sum_i \sigma_{k-1}(\lambda|i) + C,
\ee
where $C$ depends on $n$ and $\sigma_k(\lambda)$.
\end{lem}

\begin{proof}
 Let us recall the proof from \cite{Guan14,GuanSpruck}. We will use the notation $\sigma_k^{i \bar{i}} = \sigma_{k-1}(\lambda|i)$. We write
\[
\sum_i \sigma_k^{i \bar{i}} |\lambda_i| = \sigma_{k}^{r \bar{r}} |\lambda_r| + \sum_{i \neq r} \sigma_{k}^{i \bar{i}} |\lambda_i|.
\]
If $\lambda_r \geq 0$, then by concavity of $\log \sigma_k$ \cite{CNS3},
\[
\sum_i \sigma_k^{i \bar{i}} (\lambda_i - 1) \leq \sigma_k [\log \sigma_k(\lambda) - \log \sigma_k( {\bf 1})] \leq C,
\]
hence
\[
\sigma_k^{r \bar{r}} \lambda_r \leq C + \sum \sigma_k^{i \bar{i}} + \sum_{k \neq r} \sigma_k^{i \bar{i}} |\lambda_i|.
\]
It follows that
\[
\sum_{i} \sigma_k^{i \bar{i}} |\lambda_i| \leq C + \sum \sigma_k^{i \bar{i}} + 2 \sum_{i \neq r} \sigma_k^{i \bar{i}} |\lambda_i| \leq \epsilon \sum_{i \neq r} \sigma_k^{i \bar{i}} \lambda_i^2 + {C \over \epsilon} \sum \sigma_k^{i \bar{i}} + C, 
\]
which proves (\ref{quad-split}). Otherwise, we have $\lambda_r \leq 0$. In this case, $\sigma_k^{r \bar{r}} \lambda_r^2 \leq \sigma_k^{n \bar{n}} \lambda_n^2$, and
\be \label{quad-split1}
\sum_i \sigma_{k}^{i \bar{i}} |\lambda_i| \leq {\epsilon \over 2(n-1)} \sigma_k^{n \bar{n}} \lambda_n^2 + {\epsilon \over 2} \sum_{i \neq r} \sigma_k^{i \bar{i}} \lambda_i^2 + {C \over \epsilon} \sum \sigma_k^{i \bar{i}},
\ee
for some constant $C$. Since 
\[
\sum_{i \neq n} \sigma_k^{i \bar{i}} \lambda_i = k \sigma_k + \sigma_k^{n \bar{n}} |\lambda_n| \geq \sigma_k^{n \bar{n}} |\lambda_n|,
\]
we have
\[
(\sigma_k^{n \bar{n}})^2 \lambda_n^2  \leq \left( \sum_{i \neq n} \sigma_k^{i \bar{i}} \right) \left( \sum_{i \neq n} \sigma_k^{i \bar{i}} \lambda_i^2 \right).
\]
Therefore
\[
\sigma_k^{n \bar{n}} \lambda_n^2 \leq (n-1) \sum_{i \neq n} \sigma_k^{i \bar{i}} \lambda_i^2,
\]
which leads to
\[
\sigma_k^{n \bar{n}} \lambda_n^2  \leq (n-1) \sum_{i \neq r} \sigma_k^{i \bar{i}} \lambda_i^2.
\]
since $\sigma_k^{r \bar{r}} \lambda_r^2 \leq \sigma_k^{n \bar{n}} \lambda_n^2$. Substituting into (\ref{quad-split1}) proves (\ref{quad-split}).
\end{proof}

\smallskip
\par Lastly, we note a useful inequality. Let $A=[A_{\bar{k} j}]$ be an $n \times n$ Hermitian matrix with eigenvalues $(\lambda_1, \dots , \lambda_n)$, ordered such that $\lambda_1 \geq \dots \geq \lambda_n$, and let $(f_1,\dots,f_n)$ be a vector with components ordered such that $0 \leq f_1 \leq \dots \leq f_n$. Then
\be \label{schur-horn-ineq}
\sum_i f_i A_{\bar{i} i} \geq \sum_i f_i \lambda_i.
\ee
Indeed, by the Schur-Horn theorem \cite{Horn}, the vector $(A_{\bar{1} 1}, A_{\bar{2} 2}, \dots, A_{\bar{n} n})$ of the diagonal entries of $A$ is a convex combination of vectors of the form $(\lambda_{\sigma(1)}, \dots, \lambda_{\sigma(n)})$, where $\sigma$ is a permutation. Inequality (\ref{schur-horn-ineq}) can also be found in \cite{Marc}.

\section{Boundary Mixed Normal-Tangential Estimates}
The goal of this section is to prove an estimate for the mixed normal-tangential derivatives.  Before stating the result precisely, let us introduce some notation.

\subsection{Setup}
\par Let $p \in \p X$ be a point inside a boundary chart $\Omega$. We choose coordinates $z=(z^1, \cdots, z^n)$ such that $p$ corresponds to the origin and $\alpha_{\bar{k} j}(0) = \delta_{kj}$. The defining function of the boundary will be denoted by $\rho$, so that
\be \label{defining-function}
\p X \cap \Omega = \{ \rho = 0 \}, \ \ \ \Omega \subseteq \{ \rho \leq 0 \}, \ \ d \rho \neq 0 \ {\rm on} \ \p X .
\ee
Let
\[
T^{1,0}_p \p X = T_p \p X \cap J T_p \p X = \{ V \in T_p^{1,0} X : V(\rho) = 0\}.
\]
We orthogonally rotate our coordinates such that
\[
T^{1,0}_0 \p X = {\rm Span} \, \bigg\{ {\partial \over \partial z^1}, \dots, {\partial \over \partial z^{n-1}} \bigg\},
\]
and $x^n$ is in the direction of the inner normal vector at the origin, while preserving $\alpha_{\bar{k} j}(0)=\delta_{kj}$. We can then Taylor expand $\rho$ to obtain
\[
\rho = \rho_{x^n}(0) x^n + O(|z|^2).
\]
After replacing $\rho$ with ${\rho \over -\rho_{x^n}(0)}$, we obtain
\be
\rho = - x^n + O(|z|^2).
\ee
We denote $z^i = x^i + \sqrt{-1} y^i$ and
\[
t^\alpha = y^\alpha, \ \ \alpha \in \{1, \cdots, n\}, \ \ t^{n+\alpha} = x^\alpha, \ \ \alpha \in \{1, \cdots, n-1\}.
\]
By the implicit function theorem, there exists a function $\zeta(t)$ such that
\[
\rho (t, \zeta(t)) = 0 .
\]
Since $u = \underline{u}$ on $\p X$, the tangential derivatives of $u$ can be estimated. Indeed, differentiating the equation above gives the relation
\be \label{bdd-t-deriv}
\p_{t^\alpha} (u- \underline{u}) = - \p_{x^n} (u-\underline{u}) \p_{t^\alpha} \zeta
\ee
on $\p X$, and
\be \label{bdd-tangent}
\p_{t^\alpha} \p_{t^\beta} (u -\underline{u})(0) = - \p_{x^n} (u-\underline{u})(0) \rho_{t^{\alpha} t^\beta}(0)
\ee
and hence
\[
| \p_{t^\alpha} \p_{t^\beta} u(0) | \leq C,
\]
where $C$ only depends on $(X,\alpha)$, $\chi$, $\underline{u}$, by the gradient estimate (\ref{bdd-grad}). The goal of this section is to prove
\begin{prop}
In the setting of Theorem~\ref{thm-c2-uniform}, and with the above notation, there is a constant $C$, depending only on $(X,\alpha), \underline{u}, \phi$, $|\psi|_{C^{2}(X)}, \inf_{X} \psi$, and $\chi$
so that the following estimate holds
\be \label{mixed-estimate}
|h_{\bar{n} i}|(0) \leq C K^{1/2},
\ee
where we recall the notation $K = 1 + \| \nabla u \|^2_{L^\infty(X,\alpha)}$.
\end{prop}

A related estimate, with a different power of $K$ in a more general setting, is obtained in \cite{FGZ}. The power of $K^{1/2}$ here is crucial for later arguments, and this $K^{1/2}$ estimate generalizes the mixed normal-tangential estimate derived in \cite{Bouck} for complex Monge-Amp\`ere equations.

\subsection{First barrier}

\par Let $d$ denote the distance function to $\p X$. Define, for $N \gg 1$ and $c_0>0$ to be determined, the following barrier function due to B. Guan \cite{Guan98}
\be \label{v-defn}
v = (u- \underline{u}) + c_0 d - N d^2.
\ee
This barrier was also used in other works on nonlinear PDE in complex geometry, e.g. \cite{CPW,CY,FGZ,GuanLi10,GuanLi13,GuanSun,GuNgu}. We will use the notation
\[
\Omega_\delta = \Omega \cap \{ |z| < \delta \},
\]
for $\delta>0$ to be determined. The radius $\delta$ will at least be small enough such that the distance function $d$ is $C^2$ in $\Omega_\delta$. (e.g. Lemma 14.16 in \cite{GT})

\begin{lem}
  There exists $c_0,N,\delta,\tau > 0$ depending on $(X,\alpha)$, $\chi$, $\underline{u}$, $\sup_{X}\psi, \inf_{X}\psi$, such that $v: \Omega_\delta \rightarrow \mathbb{R}$ defined by (\ref{v-defn}) satisfies
\be \label{v-barrier}
\sigma_k^{p \bar{q}} \partial_p \partial_{\bar{q}} v \leq - \tau (1+\F),
\ee
and
\[
v \geq 0.
\]
\end{lem}

\par \noindent {\it Proof:} We compute
\bea \label{v-compute}
\sigma_k^{p \bar{q}} \p_p \p_{\bar{q}} v &=& \sigma_k^{p \bar{q}} \partial_p \partial_{\bar{q}} (u-\underline{u}) + c_0 \sigma_k^{p \bar{q}} \partial_p \partial_{\bar{q}} d  \nonumber\\
&&- 2 N \sigma_k^{p \bar{q}} \partial_p d \partial_{\bar{q}} d - 2 N d \sigma_k^{p \bar{q}} \partial_p \partial_{\bar{q}} d.
\eea
We will show (\ref{v-barrier}) by working at a point with coordinates such that $\alpha_{\bar{k} j} = \delta_{kj}$ and $h^i{}_j = \lambda_j \delta^i{}_j$. We start by writing
\[
\sigma_k^{p \bar{q}} \partial_p \partial_{\bar{q}} (u-\underline{u}) = \sigma_k^{p \bar{p}} (\lambda_p - (\chi_{\bar{p} p} + \underline{u}_{\bar{p} p})).
\]
The eigenvalues of $\chi + \ddb \underline{u}$ are denoted by $\underline{\lambda}$ and ordered $\underline{\lambda}_n \leq \cdots \leq \underline{\lambda}_1$. On the other hand, $\sigma_k^{1 \bar{1}}(\lambda) \leq \cdots \leq \sigma_k^{n \bar{n}}(\lambda)$. By the Schur-Horn theorem (\ref{schur-horn-ineq}),
\[
\sigma_k^{p \bar{p}} (\chi_{\bar{p} p} + \underline{u}_{\bar{p} p}) \geq \sigma_k^{p \bar{p}} \underline{\lambda}_p,
\]
and
\[
\sigma_k^{p \bar{q}} \partial_p \partial_{\bar{q}} (u-\underline{u}) \leq -\tau \F + \sigma_k^{p \bar{p}} ( \lambda_p - (\underline{\lambda}_p-\tau)),
\]
for any $\tau>0$. Choose $\tau>0$ such that
\[
\underline{\lambda} - \tau {\bf 1} \in \Gamma_k,
\]
where ${\bf 1} = ( 1, \dots, 1 )$. Since $\log \sigma_k$ is convex on $\Gamma_k$ \cite{CNS3,Garding},
\[
\sum_p {\sigma_k^{p \bar{p}} \over \sigma_k}(\lambda) [ \lambda_p - (\underline{\lambda}_p-\tau)] \leq \log \sigma_k(\lambda) - \log \sigma_k (\underline{\lambda} - \tau {\bf 1}).
\]
Therefore
\be \label{v-est1}
\sigma_k^{p \bar{q}} \partial_p \partial_{\bar{q}} (u-\underline{u}) \leq -\tau \F + C(\sup_X\psi,\underline{\lambda}).
\ee
Next, since $d$ is smooth in $\Omega_\delta$ for small enough $\delta>0$, we estimate
\be \label{v-est2}
 c_0 \sigma_k^{p \bar{q}} \partial_p \partial_{\bar{q}} d - 2N d \sigma_k^{p \bar{q}} \partial_p \partial_{\bar{q}} d \leq c_0 C \F + N d C \F.
\ee
We also estimate
\be \label{v-est3}
 - 2 N \sigma_k^{p \bar{q}} \partial_p d \partial_{\bar{q}} d \leq - 2 N \sigma_k^{1 \bar{1}} |\partial d|^2 = - {N \over 2} \sigma_k^{1 \bar{1}}, 
\ee
since $\sigma_k^{1 \bar{1}} \leq \sigma_k^{i \bar{i}}$ for any index $i$ and $|\partial d| = {1 \over 2}$ for the distance function.
\smallskip
\par Substituting (\ref{v-est1}), (\ref{v-est2}), (\ref{v-est3}) into (\ref{v-compute}), we obtain
\[
\sigma_k^{p \bar{q}} \partial_p \partial_{\bar{q}} v \leq - \tau \F - {N \over 2} \sigma_k^{1 \bar{1}} + C(c_0 + Nd) \F + C .
\]
In $\Omega_{\delta}$ we can find a constant $A$, depending only on $(X,\alpha)$ such that $d \leq A|z|$.  Then, if
\be \label{v-cond1}
(c_0 + N A\delta) \leq {\tau \over 4 C},
\ee
then
\be \label{L-v}
\sigma_k^{p \bar{q}} \partial_p \partial_{\bar{q}} v \leq - {3 \tau \over 4} \F - {N \over 2} \sigma_k^{1 \bar{1}} + C .
\ee
Next, we write
\[
- {\tau \over 2} \F - {N \over 2} \sigma_k^{1 \bar{1}} =  -\sum_i \sigma_{k-1}(\lambda|i) q_i
\]
where
\[
q = {1 \over 2} \left( \tau  +N, \tau , \dots, \tau \right).
\]
The G\aa rding inequality \cite{Garding}, which is
\[
\sum_i \sigma_{k-1}(\lambda|i) q_i \geq k \sigma_{k}(\lambda)^{(k-1)/k} \sigma_{k}(q)^{1/k},
\]
implies
\[
- {\tau \over 2} \F - {N \over 2} \sigma_k^{1 \bar{1}} \leq - k \sigma_{k}(\lambda)^{(k-1)/k} \sigma_{k}(q)^{1/k}. 
\]
We can calculate
\[
\sigma_{k}(q) = {1 \over 2^k} \bigg[  \tau^{k-1}  (\tau + N) {n-1 \choose k-1} + \tau^k {n-1 \choose k} \bigg] \geq {N \tau^{k-1} \over 2^k}.
\]
Therefore, if
\be \label{v-cond2}
N^{1/k} \geq {2 \over k} \bigg( {1 \over \tau \psi} \bigg)^{(k-1)/k} \left(C+ {\tau \over 4} \right),
\ee
then
\[
- {\tau \over 2} \F - {N \over 2} \sigma_k^{1 \bar{1}} \leq - C - {\tau \over 4}.
\]
By (\ref{L-v}), we then have
\[
\sigma_k^{p \bar{q}} \partial_p \partial_{\bar{q}} v \leq - {\tau \over 4} (1+\F).
\]
We can rename ${\tau \over 4}$ to $\tau$ to get the stated inequality (\ref{v-barrier}).
\smallskip
\par Lastly, we note that $v \geq 0$. Indeed, since $\underline{u} \leq u$ by the maximum principle, we have
\[
v \geq d (c_0 - Nd).
\]
As before we can estimate $d(z) \leq A |z|$. Therefore $v \geq 0$ in $\Omega_\delta$, provided
\be \label{v-cond3}
\delta \leq {c_0 \over AN}.
\ee
The lemma follows by choosing constants $c_0,N,\delta$ satisfying (\ref{v-cond1}), (\ref{v-cond2}) and (\ref{v-cond3}). $\qed$
\bigskip
\par

\subsection{Tangential derivatives}
Let $\alpha \in \{1,\dots, 2n-1 \}$. We define in $\Omega_\delta$ the real vector fields
\[
T_\alpha = {\partial \over \partial t^\alpha} - {\rho_{t^\alpha} \over \rho_{x^n}} {\partial \over \partial x^n}.
\]
These are tangential to the level sets of $\rho$, i.e. $T_\alpha(\rho) = 0$.
\smallskip
\par In this section, we will use the notation $\mathcal{E}$ to denote terms which can be estimated by
\[
|\mathcal{E}| \leq C  (1 + K^{1/2}) \F + C \sum_i \sigma_{k-1}(\lambda|i) |\lambda_i| + C,
\]
where as usual $C$ is only allowed to depend on $(X,\alpha)$, $\chi$, $\underline{u}$.

\begin{lem}
There exists $\delta>0$ depending on $(X,\alpha)$ such that we can estimate
\be \label{diff-3x}
\left| \sigma_k^{p \bar{q}} \partial_p \partial_{\bar{q}} \, T_\alpha (u-\underline{u}) \right| \leq {1 \over K^{1/2}} \sigma_k^{p \bar{q}} \partial_p \partial_{y^n} (u-\underline{u}) \partial_{\bar{q}} \partial_{y^n} (u-\underline{u}) + \mathcal{E}.
\ee
in $\Omega_\delta$.
\end{lem}

\par \noindent {\it Proof:} We compute 
\[
\partial_p \partial_{\bar{q}} T_\alpha (u) = \partial_p \partial_{\bar{q}} \partial_{t^\alpha} u - \p_p \p_{\bar{q}} \left( {\rho_{t^\alpha} \over \rho_{x^n}} u_{x^n} \right).
\]
Then
\bea \label{L-e-alpha}
\sigma_k^{p \bar{q}} \p_p \p_{\bar{q}} T_\alpha (u) &=& \sigma_k^{p \bar{q}} \p_p \p_{\bar{q}} \p_{t^\alpha} u - {\rho_{t^\alpha} \over \rho_{x^n}} \sigma_k^{p \bar{q}} \p_p \p_{\bar{q}} \p_{x^n} u \nonumber\\
&&- 2 {\rm Re}\left( \, \sigma_k^{p \bar{q}}\left( \p_p {\rho_{t^\alpha} \over \rho_{x^n}} \right)\left( \p_{\bar{q}} \p_{x^n} u\right)\right) - \sigma_k^{p \bar{q}}\left( \p_p \p_{\bar{q}} {\rho_{t^\alpha} \over \rho_{x^n}}\right) \p_{x^n} u.
\eea
We will first address the third order terms. For $n+1\leq \beta \leq 2n-1$ we can write
\[
{\partial \over \partial t^\beta} = {\partial \over \partial z^{\beta-n}} + {\partial \over \partial \bar{z}^{\beta-n}}
\]
and for $1\leq \beta \leq n$ we have
\[
{\partial \over \partial t^\beta} = \frac{1}{\sqrt{-1}}\left({\partial \over \partial z^{\beta-n}} - {\partial \over \partial \bar{z}^{\beta-n}}\right)
\]
Both cases are identical, so we will only treat the first case,  writing $i=\beta-n$ for simplicity. Covariantly differentiating the equation gives the relation
\[
\nabla_{t^\beta} \psi = \sigma_k^{p \bar{q}} \nabla_{t^\beta} \chi_{\bar{q} p} + \sigma_k^{p \bar{q}} \nabla_{t^\beta} \nabla_p \nabla_{\bar{q}} u.
\]
Converting covariant derivatives to partial derivatives gives
\bea
\nabla_{t^\beta} \nabla_p \nabla_{\bar{q}}  u &=& \nabla_i \nabla_p \nabla_{\bar{q}} u + \nabla_{\bar{i}} \nabla_p \nabla_{\bar{q}} u \nonumber\\
&=& \partial_p \partial_{\bar{q}} \partial_{t^\beta} u - \Gamma^{r}{}_{i p} u_{\bar{q} r} - \Gamma^{\bar{r}}{}_{\bar{i} \bar{q}} u_{\bar{r} p}.
\eea
Therefore
\be \label{triple-d}
\sigma_k^{p \bar{q}} \partial_p \partial_{\bar{q}} \partial_\beta  u = \partial_\beta \psi - \sigma_k^{p \bar{q}} \nabla_\beta \chi_{\bar{q} p} + \sigma_k^{p \bar{q}} \Gamma^r{}_{ip} u_{\bar{q} r} + \sigma_k^{p \bar{q}} \Gamma^{\bar{r}}{}_{\bar{i} \bar{q}} u_{\bar{r} p},
\ee
and so
\be \label{3rd-order}
|\sigma_k^{p \bar{q}} \partial_p \partial_{\bar{q}} \partial_{t^\beta}  u |\leq \mathcal{E}.
\ee
Here we used that for a bounded matrix $A^\ell{}_p$, we can write
\[
\sigma_k^{p \bar{q}} A^\ell{}_{p} u_{\bar{q} \ell} = \sigma_k^{p \bar{q}} A^\ell{}_{p} h_{\bar{q} \ell} - \sigma_k^{p \bar{q}} A^\ell{}_{p} \chi_{\bar{q} \ell},
\]
where $h_{\bar{k} j} = \chi_{\bar{k} j} + u_{\bar{k} j}$, and estimate
\bea \label{f_i-|lambda_i|}
| \sigma_k^{p \bar{q}} A^\ell{}_{p} h_{\bar{q} \ell}| &=& \bigg| \sum_i \sigma_{k-1}(\lambda|i) \lambda_i S^i{}_p A^p{}_q \overline{S^i{}_q} \bigg| \nonumber\\
&\leq& C \sum_i \sigma_{k-1}(\lambda|i) |\lambda_i|,
\eea
where $S = S^i{}_j$ is a unitary matrix which simultaneously diagonalizes $h^p{}_q$ with diagonal elements $\lambda_i$, and ${\partial \sigma_k \over h^p{}_i}$ with diagonal elements $\sigma_{k-1}(\lambda|i)$ (see (\ref{sigma-k-ij})).
\smallskip
\par Substituting this into (\ref{L-e-alpha}), and using that $\rho_{x^n}$ is bounded below on $\Omega_\delta$ if $\delta>0$ is taken small, it follows that
\be
\sigma_k^{p \bar{q}} \p_p \p_{\bar{q}} \, T_\alpha (u-\underline{u}) = - 2 {\rm Re} \, \sigma_k^{p \bar{q}} \p_p {\rho_{t^\alpha} \over \rho_{x^n}} \p_{\bar{q}} \p_{x^n} (u-\underline{u}) + \mathcal{E}.
\ee
We can manipulate the first term on the right-hand side by
\bea
\sigma_k^{p \bar{q}} \partial_p {\rho_{t^\alpha} \over \rho_{x^n}} \partial_{\bar{q}} \partial_{x^n} (u-\underline{u}) &=& 2 \sigma_k^{p \bar{q}} \partial_p {\rho_{t^\alpha} \over \rho_{x^n}} \partial_{\bar{q}} \partial_{n} (u-\underline{u}) \nonumber\\
&&+ \sqrt{-1} \sigma_k^{p \bar{q}} \partial_p {\rho_{t^\alpha} \over \rho_{x^n}} \partial_{\bar{q}} \partial_{y^n} (u-\underline{u}) .
\eea
By the Cauchy-Schwarz inequality and (\ref{f_i-|lambda_i|}),
\bea
& \ & \left| \sigma_k^{p \bar{q}} \partial_p {\rho_{t^\alpha} \over \rho_{x^n}} \partial_{\bar{q}} \partial_{x^n} (u-\underline{u}) \right| \nonumber\\
&\leq&  {1 \over K^{1/2}} \sigma_k^{p \bar{q}} \partial_p \partial_{y^n} (u-\underline{u}) \partial_{\bar{q}} \partial_{y^n} (u-\underline{u}) \nonumber\\
&&+  C \sum_i \sigma_{k-1}(\lambda|i) |\lambda_i| + C (1+K^{1/2})\F. 
\eea
Putting everything together, we obtain (\ref{diff-3x}). $\qed$

\subsection{Quadratic gradient term}
In constructing a barrier, we will use the term
\[
{1 \over K^{1/2}} (\p_{y^i}(u-\underline{u}))^2,
\]
for $i \in \{1, \dots, n\}$. On $\p X$, we have the relation $\rho(t,\zeta(t)) = 0$, and by (\ref{bdd-t-deriv}) we note
\[
(\p_{y^i} (u- \underline{u}))^2 = (\p_{x^n} (u-\underline{u}))^2 (\p_{y^i} \zeta)^2.
\]
Since $\p_{y^i} \zeta(0) = \p_{y^i} \rho(0) = 0$, we have $|\p_{y^i} \zeta| \leq C |t|$, and hence
\be \label{bdd-y-grad}
(\p_{y^i} (u- \underline{u}))^2 \leq C |z|^2 \ \ {\rm on} \ \p X,
\ee
by the gradient estimate (\ref{bdd-grad}). Next, we compute
\bea 
& \ & {1 \over K^{1/2}} \sigma_k^{p \bar{q}} \p_p \p_{\bar{q}} (\p_{y^i}(u-\underline{u}))^2 \nonumber\\
&=& {2 \over K^{1/2}} \sigma_k^{p \bar{q}} \p_p (u_{y^i}-\underline{u}_{y^i}) \p_{\bar{q}} (u_{y^i} - \underline{u}_{y^i}) + {2 \over K^{1/2}} (u_{y^i}-\underline{u}_{y^i}) \sigma_k^{p \bar{q}}  \p_p  \p_{\bar{q}} (u_{y^i} - \underline{u}_{y^i}) \nonumber\\
&\geq&{2 \over K^{1/2}} \sigma_k^{p \bar{q}} \p_p (u_{y^i}-\underline{u}_{y^i}) \p_{\bar{q}} (u_{y^i} - \underline{u}_{y^i}) - C \left|\sigma_k^{p \bar{q}} \p_p \p_{\bar{q}} \p_{y^i} u \right| - C \F. \nonumber
\eea
By (\ref{triple-d}), we obtain
\be \label{quad-grad0}
{1 \over K^{1/2}} \sigma_k^{p \bar{q}} \p_p \p_{\bar{q}} (\p_{y^i}(u-\underline{u}))^2 \geq {2 \over K^{1/2}} \sigma_k^{p \bar{q}} \p_p \p_{y^i} (u-\underline{u}) \p_{\bar{q}} \p_{y^i} (u - \underline{u}) + \mathcal{E}.
\ee

\subsection{Quadratic gradient term with frame}
Let $e_a = e^i{}_a \p_i$ be a local orthonormal frame of $T^{1,0} X$, such that $\{ e_a \}_{a=1}^{n-1}$ are tangential to the level sets of $\rho$. We also impose
\[
e_a(0) = {\partial \over \partial z^a}, \ \ a \in \{1, \dots, n-1 \}.
\]
We can construct the $e_a$ as follows. For $a \in \{1, \dots, n-1 \}$, let
\[
E_a = {\partial \over \partial z^a} - \bigg[ {\p_{z^a} \rho \over \p_{z^n} \rho} \bigg] {\partial \over \partial z^n}.
\]
Since $\p \rho(0) = (0,\dots, -1/2)$, these define local sections of $T^{1,0}X$ around the origin. Furthermore, they are tangential to the level sets of $\rho$ since $T_a \rho  = 0$. We may perform the Gram-Schmidt process using the metric $\alpha$ to obtain smooth sections $\{ e_a \}_{a=1}^{n-1}$ of $T^{1,0} X$ which satisfy
\[
e_a (\rho) = 0, \ \ \alpha(e_a,\overline{e_b}) = \delta_{ab}.
\]
To complete the set, we let $e_n$ be given by
\[
e_n = {E_n \over |E_n|_\alpha}, \ \ E_n = {\partial \over \partial z^n} - \sum_a \alpha(\p_n, \overline{e_a}) e_a.
\]
For fixed $a$, we consider the term
\[
\nabla_a (u-\underline{u}) \overline{\nabla_a(u-\underline{u})} = e^i{}_a (u-\underline{u})_i \overline{e^j{}_a} (u-\underline{u})_{\bar{j}}.
\]
We compute
\bea
& \ &  \sigma_k^{p \bar{q}} \partial_p \partial_{\bar{q}} \, \{ \nabla_a (u-\underline{u}) \overline{\nabla_a(u-\underline{u})} \} \nonumber\\
&=& \sigma_k^{p \bar{q}} \partial_p \partial_{\bar{q}} (e^i{}_a \overline{e^j{}_a} )(u-\underline{u})_i (u-\underline{u})_{\bar{j}} +  \sigma_k^{p \bar{q}} \partial_{\bar{q}} (e^i{}_a \overline{e^j{}_a} ) (u-\underline{u})_{pi} (u-\underline{u})_{\bar{j}} \nonumber\\
&&+  \sigma_k^{p \bar{q}} \partial_{\bar{q}} (e^i{}_a \overline{e^j{}_a} )(u-\underline{u})_i (u-\underline{u})_{\bar{j} p} + \sigma_k^{p \bar{q}} \partial_p  (e^i{}_a \overline{e^j{}_a}) (u-\underline{u})_{\bar{q} i} (u-\underline{u})_{\bar{j}} \nonumber\\
&&+ \sigma_k^{p \bar{q}}   (e^i{}_a \overline{e^j{}_a})  (u-\underline{u})_{p\bar{q}i} (u-\underline{u})_{\bar{j}}+ \sigma_k^{p \bar{q}}  (e^i{}_a \overline{e^j{}_a}) (u-\underline{u})_{\bar{q} i} (u-\underline{u})_{\bar{j} p} \nonumber\\
&&+ \sigma_k^{p \bar{q}} \partial_p  (e^i{}_a \overline{e^j{}_a}) (u-\underline{u})_{i} (u-\underline{u})_{\bar{q}\bar{j}}+ \sigma_k^{p \bar{q}}  e^i{}_a \overline{e^j{}_a}  (u-\underline{u})_{pi}  (u-\underline{u})_{\bar{q}\bar{j}} \nonumber\\
&&+ \sigma_k^{p \bar{q}} e^i{}_a \overline{e^j{}_a} (u-\underline{u})_{i} (u-\underline{u})_{p\bar{q}\bar{j}}. \nonumber
\eea
By (\ref{3rd-order}) and (\ref{f_i-|lambda_i|}), we may group terms as
\bea \label{quad-grad1}
& \ & {1 \over K^{1/2}} \sigma_k^{p \bar{q}} \p_p \p_{\bar{q}} |\nabla_a (u-\underline{u})|^2 \nonumber\\
&=& {1 \over K^{1/2}} \sigma_k^{p \bar{q}} e^i{}_a (u-\underline{u})_{\bar{q} i} \overline{e^j{}_a} (u- \underline{u})_{\bar{j} p} + {1 \over K^{1/2}} \sigma_k^{p \bar{q}} e^i{}_a (u-\underline{u})_{p i} \overline{e^j{}_a} (u- \underline{u})_{\bar{j} \bar{q}} \nonumber\\
&& +  {2 \over K^{1/2}} {\rm Re} \, \sigma_k^{p \bar{q}} \p_{\bar{q}} (e^i{}_a \overline{e^j{}_a}) (u-\underline{u})_{ip} (u-\underline{u})_{\bar{j}}  + \mathcal{E} .
\eea
We start with the first term of (\ref{quad-grad1}). We can write
\[
\sigma_k^{p \bar{q}} e^i{}_a (u-\underline{u})_{\bar{q} i} \overline{e^j{}_a} (u- \underline{u})_{\bar{j} p} = \sigma_k^{p \bar{q}} (h_{\bar{q} a} - \underline{h}_{\bar{q} a}) (h_{\bar{a} p} - \underline{h}_{\bar{a} p}), 
\]
where we use the index $a$ for the frame direction $e_a$, and the notation $\underline{h}_{\bar{k} j} = \chi_{\bar{k} j} + \underline{u}_{\bar{k} j}$. By the Cauchy-Schwarz inequality,
\be \label{quad-grad2}
\sigma_k^{p \bar{q}} e^i{}_a (u-\underline{u})_{\bar{q} i} \overline{e^j{}_a} (u- \underline{u})_{\bar{j} p}  \geq {1 \over 2} \sigma_k^{p \bar{q}} h_{\bar{q} a} h_{\bar{a} p} + \mathcal{E}. 
\ee
The second term of (\ref{quad-grad1}) is nonnegative and we will leave it for now. Next, we work on the third term of (\ref{quad-grad1}). Since
\[
{\partial \over \partial z^i} = {\partial \over \partial \bar{z}^i} - \sqrt{-1} {\partial \over \partial y^i},
\]
we can manipulate the term
\bea
& \ &
{2 \over K^{1/2}} \sigma_k^{p \bar{q}} \p_{\bar{q}} (e^i{}_a \overline{e^j{}_a}) (u-\underline{u})_{ip} (u-\underline{u})_{\bar{j}}  \nonumber\\
&=& {2 \over K^{1/2}} \sigma_k^{p \bar{q}} \p_{\bar{q}} (e^i{}_a \overline{e^j{}_a}) (u-\underline{u})_{\bar{i} p} (u-\underline{u})_{\bar{j}}  - \sqrt{-1} {2 \over K^{1/2}} \sigma_k^{p \bar{q}} \p_{\bar{q}} (e^i{}_a \overline{e^j{}_a}) (u-\underline{u})_{y^i p} (u-\underline{u})_{\bar{j}} . \nonumber
\eea
By the Cauchy-Schwarz inequality and Young's inequality,
\bea
& \ & {2 \over K^{1/2}} \sum_{i,j} \left| (u-\underline{u})_{\bar{j}} \sigma_k^{p \bar{q}} \p_p \p_{y^i} (u-\underline{u}) \p_{\bar{q}} (e^i{}_a \overline{e^j{}_a})  \right|\nonumber\\
&\leq& 2 \sum_{i,j} \left( \sigma_k^{p \bar{q}} \p_p \p_{y^i} (u-\underline{u}) \p_{\bar{q}} \p_{y^i} (u-\underline{u}) \right)^{1/2} \left( \sigma_k^{p \bar{q}} \p_p (e^i{}_a \overline{e^j{}_a}) \p_{\bar{q}} (\overline{e^i{}_a} e^j{}_a) \right)^{1/2} \nonumber\\
&\leq& {1 \over (n-1) K^{1/2}} \sum_{i=1}^n \sigma_k^{p \bar{q}} \p_p \p_{y^i} (u-\underline{u}) \p_{\bar{q}} \p_{y^i} (u-\underline{u}) + \mathcal{E}. \nonumber
\eea
Therefore
\bea \label{quad-grad3}
& \ &  {2 \over K^{1/2}} {\rm Re} \, \sigma_k^{p \bar{q}} \p_{\bar{q}} (e^i{}_a \overline{e^j{}_a}) (u-\underline{u})_{ip} (u-\underline{u})_{\bar{j}}  \nonumber\\
&\geq& - {1 \over (n-1) K^{1/2}} \sum_{i=1}^n \sigma_k^{p \bar{q}} \p_p \p_{y^i} (u-\underline{u}) \p_{\bar{q}} \p_{y^i} (u-\underline{u})  + \mathcal{E}.
\eea
In (\ref{quad-grad1}), we apply (\ref{quad-grad2}) on the first term, drop the second term, and apply (\ref{quad-grad3}) on the third term. We are left with
\bea \label{grad-squared-ineq}
& \ & \sigma_k^{p \bar{q}} \p_p \p_{\bar{q}} {1 \over K^{1/2}} \sum_{a=1}^{n-1} |\nabla_a (u-\underline{u})|^2 \\
&\geq& {1 \over 2 K^{1/2}} \sum_{a=1}^{n-1} \sigma_k^{p \bar{q}} h_{\bar{q} a} h_{\bar{a} p} - {1 \over K^{1/2}} \sum_{i=1}^n \sigma_k^{p \bar{q}} \p_p \p_{y^i} (u-\underline{u}) \p_{\bar{q}} \p_{y^i} (u-\underline{u})  + \mathcal{E}. \nonumber
\eea
We will now study the positive term
\[
\sum_{a=1}^{n-1} \sigma_k^{p \bar{q}} h_{\bar{q} a} h_{\bar{a} p}.
\]
At a point, we take new coordinates such that $\alpha_{\bar{k} j} = \delta_{kj}$ and $h_{\bar{k} j} = \lambda_j \delta_{kj}$. We express the frame $e_a = e^i{}_a \partial_i$ using these coordinates and also use these coordinates for the contracted indices $p,q$ above. The term becomes
\[
\sum_{a=1}^{n-1} \sum_{i=1}^n |e^i{}_a|^2 \sigma_{k-1}(\lambda|i) \lambda_i^2.
\]
Since the frame $e_a$ is unitary and coordinates $z^i$ are chosen such that the metric is the identity, we have
\[
\sum_{a=1}^n |e^i{}_a|^2 = 1, \ \ \sum_{i=1}^n |e^i{}_a|^2 = 1.
\]
We can therefore write
\[
\sum_{a=1}^{n-1} \sigma_k^{p \bar{q}} h_{\bar{q} a} h_{\bar{a} p} = \sum_{i=1}^{n} \sigma_{k-1}(\lambda|i) \lambda_i^2 (1- |e^i{}_n|^2).
\]
Since the $|e^i{}_n|^2$ sum to one, there exists an index $r$ such that $(1/n) \leq |e^i{}_r|^2 \leq 1$. It follows that all $i \neq r$ satisfy
\[
|e^i{}_n|^2 \leq 1 - {1 \over n}.
\]
Therefore
\[
\sum_{a=1}^{n-1} \sigma_k^{p \bar{q}} h_{\bar{q} a} h_{\bar{a} p} \geq {1 \over n}\sum_{i \neq r} \sigma_{k-1}(\lambda|i) \lambda_i^2.
\]
Going back to (\ref{grad-squared-ineq}), we conclude
\bea \label{quad-grad4}
& \ & \sigma_k^{p \bar{q}} \p_p \p_{\bar{q}} {1 \over K^{1/2}}\sum_{a=1}^{n-1} |\nabla_a (u-\underline{u})|^2 \\
&\geq& {1 \over 2 n K^{1/2}} \sum_{i \neq r} \sigma_{k-1}(\lambda|i) \lambda_i^2 - {1 \over K^{1/2}} \sum_{i=1}^n \sigma_k^{p \bar{q}} \p_p \p_{y^i} (u-\underline{u}) \p_{\bar{q}} \p_{y^i} (u-\underline{u})  + \mathcal{E}. \nonumber
\eea

\subsection{Final barrier}
For constants $A,B \gg 1$ to be determined, let
\bea
\Psi &=& A K^{1/2} v + B K^{1/2} |z|^2  - {1 \over K^{1/2}} \sum_{i=1}^n (\p_{y^i} (u-\underline{u}))^2 \nonumber\\
&&- {1 \over K^{1/2}} \sum_{a=1}^{n-1} |\nabla_a (u-\underline{u})|^2. \nonumber
\eea
Combining  (\ref{v-barrier}), (\ref{diff-3x}), (\ref{quad-grad0}), (\ref{quad-grad4}), and cancelling terms we obtain
\bea
& \ & \sigma_k^{p \bar{q}} \p_p \p_{\bar{q}} (\Psi + T_\alpha (u-\underline{u})) \nonumber\\
&\leq& - A \tau K^{1/2} (1+\F) + B K^{1/2} \sum_p \sigma_k^{p \bar{p}} - {1 \over 2nK^{1/2}} \sum_{i \neq r} \sigma_{k-1}(\lambda|i) \lambda_i^2 \nonumber\\
&& -{1 \over K^{1/2}} \sum_{i=1}^{n-1}\sigma_k^{p \bar{q}} \p_p \p_{y^i} (u-\underline{u}) \p_{\bar{q}} \p_{y^i} (u-\underline{u}) \nonumber\\
&&+C  (1 + K^{1/2}) \F + C \sum_i \sigma_{k-1}(\lambda|i) |\lambda_i| + C.
\eea
For $A \geq C B \tau^{-1} + A_0 \tau^{-1}$ with $A_0$ large enough, we then have
\bea
\sigma_k^{p \bar{q}} \p_p \p_{\bar{q}} (\Psi + T_\alpha (u-\underline{u})) &\leq& -{A_0 \over 2} K^{1/2} (1+\F) - {1 \over 2nK^{1/2}} \sum_{i \neq r} \sigma_{k-1}(\lambda|i) \lambda_i^2 \nonumber\\
&& + C \sum \sigma_{k-1}(\lambda|i)|\lambda_i|.
\eea
By applying (\ref{quad-split}) with $\epsilon ={1 \over 2n C K^{1/2}}$, we see that for $A_0 \gg 1$, then
\[
\sigma_k^{p \bar{q}} \p_p \p_{\bar{q}} (\Psi + T_\alpha(u-\underline{u})) \leq 0.
\]
Next, we look at the values of $\Psi + T_\alpha(u-\underline{u})$ on the boundary of $\Omega_\delta = \Omega \cap B_\delta(0)$. This boundary has two pieces.
\smallskip
\par On the piece $\p X \cap \Omega_\delta$, we have (\ref{bdd-y-grad}) and
\[
T_\alpha(u-\underline{u})=0, \ \ \nabla_a (u-\underline{u}) = 0.
\]
Therefore
\[
\Psi + T_\alpha(u-\underline{u}) \geq A K^{1/2} v + BK^{1/2} |z|^2 - C |z|^2 \geq 0,
\]
when $B$ is taken large enough.
\smallskip
\par On the piece $\p B_\delta \cap \Omega_\delta$, we have
\[
\Psi + T_\alpha(u-\underline{u}) = A K^{1/2} v + B K^{1/2} \delta^2 - CK^{1/2} \geq 0,
\]
for $B$ large.
\smallskip
\par It follows that
\[
\Psi + T_\alpha(u-\underline{u}) \geq 0 \ {\rm on} \ \partial \Omega_\delta.
\]
By the maximum principle,
\[
\Psi + T_\alpha(u-\underline{u}) \geq 0 \ {\rm on} \ \Omega_\delta, \ \ [\Psi + T_\alpha(u-\underline{u})] (0)= 0.
\]
Therefore
\[
\p_{x^n} [\Psi + T_\alpha(u-\underline{u})](0) \geq 0.
\]
It follows that
\[
0 \leq A K^{1/2} \p_{x^n} v(0) - (\p_{x^n} {\rho_{t^\alpha} \over \rho_{x^n}})(0) \p_{x^n} (u-\underline{u})(0) + \p_{x^n} \p_{t^\alpha} (u - \underline{u})(0) .
\]
Since $|\p_{x^n} v| \leq C$ on $\p X$ by the boundary gradient estimate (\ref{bdd-grad}), we conclude
\[
\p_{x^n} \p_{t^\alpha} u (0) \geq -CK^{1/2}.
\]
We can apply the same argument to $\Psi - T_\alpha(u-\underline{u})$. It follows that
\[
|\p_{x^n} \p_{t^\alpha} u|(0) \leq CK^{1/2}.
\]
Let $i \in \{1, \dots, n-1 \}$. Since
\[
u_{\bar{n} i} = {1 \over 4} \left( {\partial \over \partial x^n} + \sqrt{-1} {\partial \over \partial y^n} \right) \left( {\partial \over \partial t^{n+\alpha}} + \sqrt{-1} {\partial \over \partial t^\alpha} \right) u
\]
for some $\alpha \in \{1 , \dots, n-1\}$, it follows that
\[
|h_{\bar{n} i}|(0) \leq C K^{1/2},
\]
which is the mixed normal-tangential estimate that we will need in the following section.

\section{Boundary Double Normal Estimate}
\par Let $p \in \p X$ be a boundary point, with coordinates $z=(z^1, \cdots, z^n)$, $z^i = x^i + \sqrt{-1} y^i$ such that $p$ corresponds to the origin. Take $z$ to be coordinates centered at the origin such that $\alpha_{\bar{k} j}(0)= \delta_{kj}$, and rotate them such that ${\partial \over \partial x^n}$ is the inner normal vector at $p$. It remains to estimate
\[
|u_{\bar{n} n}| = \left|{\p \over \partial z^n} {\partial \over \partial \bar{z}^n} u \right|.
\]
We may further rotate coordinates in the tangential directions, such that at $p$ the matrix $h_{\bar{k} j} = \chi_{\bar{k} j} + u_{\bar{k} j}$ is of the form
\be \label{h-bdd-coords}
h = \begin{bmatrix}
h_{\bar{n} n} & h_{\bar{n} 1} & h_{\bar{n} 2} & \cdots & h_{\bar{n} n-1}\\
h_{\bar{1} n} & \lambda'_1 & 0 & \cdots&0\\
h_{\bar{2} n} & 0 & \lambda'_2 &\cdots& 0\\
\vdots & \vdots& \ddots& \vdots&\vdots\\
h_{\overline{n-1} n} &0& 0 & \cdots & \lambda'_{n-1}
\end{bmatrix}.
\ee
Since the eigenvalues of $h$ are in the $\Gamma_{k}$ cone, we know that
\[
h_{\bar{n} n} + \sum_{i=1}^{n-1} \lambda'_i \geq 0.
\]
It follows from the double-tangential estimate that
\[
h_{\bar{n} n} \geq -C.
\]
It remains to estimate $h_{\bar{n} n}$ from above to obtain the estimate on $|u_{\bar{n} n}|$. The main estimate for this purpose will be the following.

\begin{thm} \label{thm-lower-bdd}
Let $\lambda' \in \mathbb{R}^{n-1}$ be the eigenvalues of the endomorphism $h$ restricted to the subbundle $T^{1,0} \p X$. Then
\[
\sigma_{k-1}(\lambda') \geq \kappa_0 >0 ,
\]
where $\kappa_0>0$ depends on $(X,\alpha)$, $\chi$, $\psi$, $\underline{u}$.
\end{thm}

Assuming this, we may prove the upper bound for $h_{\bar{n} n}$. Suppose $p \in \p X$ is a point where $h_{\bar{n} n} \geq 0$ with coordinates chosen as in (\ref{h-bdd-coords}). Since $\sigma_k(h)$ is defined as the coefficient of $\theta^k$ in the expansion $\det(I_n + \theta h)$, we have
\[
\psi = \sigma_k(h) = h_{\bar{n} n} \sigma_{k-1}(\lambda') - \sum_i |h_{\bar{n} i}|^2 \sigma_{k-2} (\lambda'|i) + \sigma_k(\lambda').
\]
Applying Theorem \ref{thm-lower-bdd}, 
\[
h_{\bar{n} n} \leq \kappa_0^{-1} (\psi + \sum_i |h_{\bar{n} i}|^2 \sigma_{k-2} (\lambda'|i) - \sigma_k(\lambda')).
\]
Our tangential and mixed normal-tangential (\ref{mixed-estimate}) estimates give $|\lambda'|\leq C$ and $|h_{\bar{n} i}| \leq C K^{1/2}$. Therefore
\[
h_{\bar{n} n} \leq CK.
\]
This gives the boundary $C^2$ estimate. Combining this with the estimate (\ref{C2-interior}) in the interior (due to \cite{HMW,Sz}), it follows that

\begin{prop}\label{prop: c2-est}
In the setting of Theorem~\ref{thm-c2-uniform}, there is a constant $C$, depending only on $(X,\alpha), \underline{u}, \phi, |\psi|_{C^{2}(X)}, \inf_{X} \psi$, and $\chi$ so that
\be \label{c2-est}
\sup_X \| \ddb u \|_{(X,\alpha)} \leq CK.
\ee
\end{prop}

\bigskip
\par \noindent {\it Proof of Theorem \ref{thm-lower-bdd}:} By Proposition~\ref{prop-c2-est-HMW} it suffices to prove~\eqref{c2-est} on $\del X$.  From now on, we assume $k \geq 2$, since the case $k=1$ is trivial. Let $0 \in \p X$ be a point with  coordinates such that the metric $\alpha_{\bar{k} j}(0) = \delta_{kj}$. Let $\Omega_\delta = X \cap B_\delta(0)$ for $\delta>0$ to be chosen later, and $B_\delta(0)$ the ball of radius $\delta$ at the origin. Orthogonally rotate coordinates such that $T_0^{1,0}(\p X)$ is spanned by ${\partial \over \partial z^\alpha}$ for $\alpha \in \{1, \dots , n-1 \}$, and $x^n$ is in the direction of the inner normal vector at the origin.
\smallskip
\par We will use Greek indices $\alpha,\beta \in \{1, \dots, n-1 \}$ for tangential directions. We will also use $\lambda'= (\lambda'_1, \dots,\lambda'_{n-1})$ for the eigenvalues of $(\chi_{\bar{\alpha} \beta} + u_{\bar{\alpha} \beta})(0)$ and $\underline{\lambda'} = (\underline{\lambda'_1}, \dots, \underline{\lambda'_{n-1}})$ for the eigenvalues of $(\chi_{\bar{\alpha} \beta} + \underline{u}_{\bar{\alpha} \beta})(0)$. We note that
\[
\lambda',\underline{\lambda'} \in \Gamma_{k-1} \subset \mathbb{R}^{n-1}.
\]
This is because the eigenvalues $\lambda$ of $h_{\bar{k} j}(0)$ are in the $\Gamma_k$ cone, and hence by the Schur-Horn theorem \cite{Horn}, we also have
\be \label{schur-horn}
(h_{\bar{n} n}(0),\lambda'_1, \dots,\lambda'_{n-1}) \in \Gamma_k, \ \ \sigma_k(h_{\bar{n} n},\lambda') \geq \sigma_k(\lambda).
\ee
It follows from (\ref{ellipticity}) that
\[
( \lambda'_1, \cdots, \lambda'_{n-1} ) \in \Gamma_{k-1},
\]
and a similar argument shows $\underline{\lambda'} \in \Gamma_{k-1}$.
\smallskip
\par Let $\rho \leq 0$ be as before the defining function (\ref{defining-function}) of the boundary in these coordinates, expressed as
\be \label{rho-expression}
\rho = - x^n + O(|z|^2).
\ee
Since $u=\underline{u}$ on $\p X$, as noted previously (\ref{bdd-tangent}), we have
\be \label{key-bdd-relation}
(\chi_{\bar{\alpha} \beta} + u_{\bar{\alpha} \beta})(0) = (\chi_{\bar{\alpha} \beta}+ \underline{u}_{\bar{\alpha} \beta})(0) - (u-\underline{u})_{x^n}(0) \rho_{\bar{\alpha} \beta}(0).
\ee
Since $u \geq \underline{u}$ and $u=\underline{u}$ on $\p X$, we know that
\[
- (u-\underline{u})_{x^n}(0) \leq 0.
\]
Suppose first that $(u-\underline{u})_{x^n}(0) = 0$. Then at the origin, we have $\lambda' = \underline{\lambda'}$, and hence
\[
\sigma_{k-1}(\lambda') = \sigma_{k-1}(\underline{\lambda'}).
\]
If $\underline{\lambda}$ are the eigenvalues of $(\chi + \ddb \underline{u}) (0)$, then using the estimate (\ref{sigma11-lower-bdd}) and (\ref{schur-horn}), we see that
\be \label{subsoln-lower-bdd}
\sigma_{k-1}(\underline{\lambda'}) \geq {k \over n} {\psi \over \underline{\lambda_1}} > 0,
\ee
which proves Theorem \ref{thm-lower-bdd}. We may thus assume
\[
\eta :=  (u-\underline{u})_{x^n}(0) > 0.
\]
We note that
\be \label{eta-bound}
\eta \leq | \nabla (u-\underline{u})|(0) \leq C,
\ee
since $|\nabla u|$ is bounded on $\p X$ as noted in Lemma~\ref{lem:c0-c1-bdry}, see e.g. (\ref{bdd-grad}).
\smallskip
\par With this setup in place, we now follow the technique developed by Caffarelli-Nirenberg-Spruck \cite{CNS3}. Since we are dealing with the complex Hessian, we will also use some ideas and notation of S.-Y. Li \cite{LiSY}.
\smallskip
\par For a real parameter $t$, consider the family of $(n-1) \times (n-1)$ matrices
\be \label{A_t-defn}
A_t = t (\chi_{\bar{\alpha} \beta} + \underline{u}_{\bar{\alpha} \beta})(0) - \eta \rho_{\bar{\alpha} \beta}(0).
\ee
Since $A_1 = (\chi_{\bar{\alpha} \beta} + u_{\bar{\alpha} \beta})(0)$, we know that at $t=1$ we have
\[
\lambda(A_1) \in \Gamma_{k-1} \subset \mathbb{R}^{n-1}.
\]
For $t \rightarrow - \infty$, we have $\lambda(A_t) \notin \Gamma_{k-1}$. Let $t_0 < 1$ denote the first time when the eigenvalues hit the boundary of the cone;
\[
\lambda(A_{t_0}) \in \partial \Gamma_{k-1}.
\]
Then by definition, there is an $\ell \in \{1, \dots, k-1\}$ such that we have
\be
\sigma_\ell(A_{t_0})=0.
\ee
By the Newton-Maclaurin inequality $({\sigma_r \over C^n_k})^{1/r} \leq ({\sigma_s \over C^n_s})^{1/s}$ for $s \leq r$, we conclude 
\[
\sigma_{k-1}(A_{t_0}) = 0.
\]
Our goal is to separate $t_0$ away from $t=1$ and obtain the estimate
\be \label{t-bdd}
t_0 \leq 1- \kappa,
\ee
for a uniform $\kappa>0$.
\smallskip
\par We now explain why this will imply the lower bound of Theorem \ref{thm-lower-bdd}. By (\ref{key-bdd-relation}), we can write
\[
(\chi_{\bar{\alpha} \beta} + u_{\bar{\alpha} \beta})(0) = (1-t_0) (\chi_{\bar{\alpha} \beta}+ \underline{u}_{\bar{\alpha} \beta})(0) + A_{t_0}.
\]
By concavity \cite{CNS3} and homogeneity of $\sigma_{k-1}^{1/(k-1)}$ on the space of matrices with eigenvalues in $\Gamma_{k-1}$, at the point $0$ we have
\bea
\sigma_{k-1}^{1/(k-1)}[\chi_{\bar{\alpha} \beta} + u_{\bar{\alpha} \beta}] &\geq& (1-t_0)\sigma_{k-1}^{1/(k-1)}[\chi_{\bar{\alpha} \beta}+ \underline{u}_{\bar{\alpha} \beta}] + \sigma_{k-1}^{1/(k-1)} [ A_{t_0} ] \nonumber\\ 
&=& (1-t_0)\sigma_{k-1}^{1/(k-1)}[\chi_{\bar{\alpha} \beta}+ \underline{u}_{\bar{\alpha} \beta}]. \nonumber
\eea
Thus (\ref{t-bdd}) implies
\[
\sigma_{k-1}^{1/(k-1)}[\chi_{\bar{\alpha} \beta} + u_{\bar{\alpha} \beta}] \geq \kappa \sigma_{k-1}^{1/(k-1)}[\chi_{\bar{\alpha} \beta}+ \underline{u}_{\bar{\alpha} \beta}] . 
\]
By the lower bound for the subsolution (\ref{subsoln-lower-bdd}), it follows that
\[
\sigma_{k-1}^{1/(k-1)}[\chi_{\bar{\alpha} \beta} + u_{\bar{\alpha} \beta}] \geq \kappa_0 >0,
\]
proving Theorem \ref{thm-lower-bdd}.
\bigskip
\par To prove (\ref{t-bdd}), we will use the following auxillary functions of Caffarelli-Nirenberg-Spruck \cite{CNS3}, defined on $\Omega_\delta$:
\[
D(z) = - \rho(z) + \tau |z|^2 \geq 0,
\]
\[
w (z) = \underline{u}(z) - {\eta \over t_0} \rho(z) + (\ell_i z^i + \ell_{\bar{i}} \bar{z}^i) \rho(z) + M D(z)^2 
\]
\[
\Psi(z) = w(z) + \epsilon (|z|^2 - {1 \over C_0} x^n).
\]
The parameters will be chosen as follows: $\tau,\epsilon>0$ will be small constants, $M,C_0>1$ will be large constants, $\ell_i \in \mathbb{C}$ and $\ell_{\bar{i}} = \overline{\ell_i}$. The goal is

\begin{lem} \label{lem-nn-barrier}
Suppose $t_0 \geq 1/2$. There exists parameters $\delta,\tau,\epsilon,M,C_0,\ell_i$ depending only on $(X,\alpha)$, $\chi$, $\inf_{X}\psi$, $\underline{u}$, such that
\[
u(z) \leq \Psi(z)
\]
on $\Omega_\delta$.
\end{lem}

Given this, we prove the $t_0$ estimate (\ref{t-bdd}) and hence Theorem \ref{thm-lower-bdd}. Since $\Psi(0)=\underline{u}(0)=u(0)$, we have
\[
\p_{x^n} \Psi(0) \geq \p_{x^n} u(0),
\]
Using the definition of $\Psi$, this gives
\[
 -(u-\underline{u})_{x^n}(0) \geq {\epsilon \over C_0} + {\eta \over t_0} \p_{x^n} \rho (0). 
\]
Since $\eta = (u-\underline{u})_{x^n}(0)$ and $\p_{x^n} \rho(0) = -1$, this inequality can be rearraged to
\[
{1 \over t_0} \geq 1 + {\epsilon \over C_0 \eta}.
\]
Hence,
\[
t_0  \leq {1 \over 1+ \epsilon \eta^{-1} C_0^{-1}} := 1-\kappa,
\]
for uniform
\[
\kappa ={ \epsilon C_0^{-1} \over \eta + \epsilon C_0^{-1}} > 0,
\]
which proves (\ref{t-bdd}). Recall that $0< \eta \leq C$ is bounded above by Lemma~\ref{lem:c0-c1-bdry}, as noted following (\ref{eta-bound}).

\bigskip
\par \noindent {\it Proof of Lemma \ref{lem-nn-barrier}:} We introduce some more notation before starting the estimate. We may perform an orthogonal change of coordinates in the tangential directions at the origin to arrange that $A_{t_0}$ (\ref{A_t-defn}) is diagonal. Let $\tilde{\lambda} = (\tilde{\lambda}_1, \dots, \tilde{\lambda}_{n-1})$ denote the eigenvalues of $A_{t_0}$, arranged as usual so that
\[
\tilde{\lambda}_1 \geq \dots \geq \tilde{\lambda}_{n-1}.
\]
Since $\tilde{\lambda} \in \p \Gamma_{k-1}$, we take $\ell \in \{1, \dots, k-1 \}$ to be the smallest integer such that $\sigma_\ell(\tilde{\lambda})=0$. We have the equation
\[
\sum_{i=1}^{n-1} \sigma_{\ell-1}(\tilde{\lambda}|i) \tilde{\lambda}_i = \ell \sigma_{\ell}(\tilde{\lambda}) = 0.
\]
The supporting hyperplane of the cone $\Gamma_{\ell}$ at the point $\tilde{\lambda} \in \p \Gamma_{\ell}$ has normal vector proportional to
\[
(\sigma_{\ell-1}(\tilde{\lambda}|1), \dots , \sigma_{\ell-1}(\tilde{\lambda}|n-1)) \in \mathbb{R}^{n-1}.
\]
Note that $\sigma_{\ell-1}(\tilde{\lambda}|i) \ne 0$ for some $1 \leq i \leq n-1$ for otherwise $\sigma_{\ell-1}(\tilde{\lambda})=0$ contradicting our choice of $\ell$. Denote
\[
\nu_i = {1 \over \sum \sigma_{\ell-1}(\tilde{\lambda}|i)} \sigma_{\ell-1}(\tilde{\lambda}|i).
\]
Since $\Gamma_\ell$ is a convex set, we have
\be \label{hyperplane}
\Gamma_{\ell} \subset \{ \lambda \in \mathbb{R}^{n-1} : \sum_\alpha \nu_\alpha \lambda_\alpha > 0 \}.
\ee
We also note
\[
\sum_\alpha \nu_\alpha = 1, \ \ \nu_{n-1} \geq \dots \geq \nu_1 \geq 0, \quad \text{and} \quad \sum_{\alpha=1}^{n-1} \nu_\alpha \tilde{\lambda}_\alpha = 0.
\]
Let
\[
\xi_a = \sum_{i=1}^n \xi^i{}_a {\partial \over \partial z^i}
\]
be a local frame for $T^{1,0} X$ defined in $\Omega_\delta$, which is orthonormal with respect to the metric $\alpha$, and such that $\{ \xi_a \}_{a=1}^{n-1}$ span the holomorphic tangent space of the level sets $\{ D(z) = const \}$. We also arrange this frame such that $\xi_a(0) = {\partial \over \partial z^a}$ for $a \in \{1, \dots, n-1 \}$. To see such a frame exists, define the $(1,0)$ form $\del D$, and consider the map
\[
\del D: T^{1,0}X\rightarrow \mathbb{C}
\]
Since $\del D \ne 0$ (close enough to $0$), the kernel of this map defines a $n-1$ dimensional subbundle $V\subset T^{1,0}X$, which inherits the metric $\alpha|V$.  Now take a smooth, unitary frame for this bundle.
\smallskip
\par Define a local operator on $(1,1)$ forms as follows; if $\beta = \sqrt{-1} \beta_{\bar{k} j} dz^j \wedge d \bar{z}^k$ by
\[
\Lambda_\nu \beta = {1 \over \sqrt{-1}} \sum_{a=1}^{n-1} \nu_a \beta(\xi_a,\overline{\xi_a}) = \sum_{a=1}^{n-1} \nu_a \xi^j{}_a \overline{\xi^k{}_a} \beta_{\bar{k} j}.
\]

The main estimate for the barrier function $w$ is the following:

\begin{lem} \label{lemma-Lambda-w} Let $1/2 \leq t_0 \leq 1$. There exists parameters $\tau$, $M$, $\ell_i$, $\delta$, depending on $(X,\alpha)$, $\chi$, $\inf_X\psi$, $\underline{u}$, such that
\[
\Lambda_\nu (\chi + \ddb w) \leq 0.
\]
in $\Omega_\delta$.
\end{lem}

\par \noindent {\it Proof:} By the definition of $A_{t_0}$, which we assume is diagonal, we have
\be \label{boundary-cone-id}
0 = \sum_\alpha \nu_\alpha \tilde{\lambda}_\alpha = t_0 \sum_\alpha \nu_\alpha (\chi_{\bar{\alpha} \alpha} + \underline{u}_{\bar{\alpha} \alpha}) - \eta \sum_\alpha \nu_\alpha \rho_{\bar{\alpha} \alpha}.
\ee
We know $\underline{\lambda'} \in \Gamma_{k-1} \subseteq \Gamma_\ell$. As currently defined, $\underline{\lambda'}$ corresponds to the eigenvalues of $(\chi_{\bar{\alpha} \beta} + u_{\bar{\alpha} \beta})(0)$. At other points, we consider the vector bundle $E \subset T^{1,0} \Omega_\delta$ defined by the kernel of the map
\[
\del \rho: T^{1,0}X\rightarrow \mathbb{C},
\]
which is non-degenerate in $\Omega_\delta$. Using the metric $\alpha|_E$ and the orthogonal projection $\pi: T^{1,0} \Omega_\delta \rightarrow E$, we get an induced hermitian endormorphism $h_{E}: E \rightarrow E$. We can then define $\underline{\lambda'}(p)$ to be the eigenvalues of this map.
\smallskip
\par By the lower bound for the subsolution (\ref{subsoln-lower-bdd}), the vector $\underline{\lambda'}$ lives in a bounded region and stays a fixed distance away from the boundary of the cone $\Gamma_\ell$. Therefore
\[
\sum_\alpha \nu_\alpha \underline{\lambda'_\alpha}(0) \geq \inf_{p \in \p X} \sum_\alpha \nu_\alpha \underline{\lambda'_\alpha}(p) \geq a > 0,
\]
where $a$ depends on $(X,\alpha)$ and $\chi+ \ddb \underline{u}$. By the Schur-Horn theorem (\ref{schur-horn-ineq}),
\[
\sum_{\alpha} \nu_\alpha (\chi_{\bar{\alpha} \alpha} + \underline{u}_{\bar{\alpha} \alpha}) \geq \sum_\alpha \nu_\alpha \underline{\lambda'_\alpha} \geq a.
\]
Therefore, by (\ref{boundary-cone-id}),
\[
0 \geq a t_0 - \eta \sum_\alpha \nu_\alpha \rho_{\bar{\alpha} \alpha}.
\]
Since we assume $t_0 \geq 1/2$, and $\eta$ is bounded above (\ref{eta-bound}), we have
\be \label{a-1}
\sum_{\alpha=1}^{n-1} \nu_\alpha \rho_{\bar{\alpha} \alpha} \geq a_1 > 0.
\ee
This is the key good term which will carry the barrier argument through. 
\smallskip
\par We now start computing the quantity stated in the lemma. First, we have
\bea
\partial_j \partial_{\bar{k}} w &=& \underline{u}_{\bar{k} j} - {\eta \over t_0} \rho_{\bar{k} j} + \ell_{\bar{k}} \rho_j + \ell_j \rho_{\bar{k}} \nonumber\\
&&+ (\ell_i z^i + \ell_{\bar{i}} \bar{z}^i) \rho_{\bar{k} j} + 2M \p_j D \p_{\bar{k}} D + (2M D) \p_j \p_{\bar{k}} D.
\eea
Therefore
\[
\Lambda_\nu (\chi + \ddb w) = T_1 + T_2 + T_3 + T_4,
\]
where
\bea
T_1 &=& {1 \over t_0} \sum_{a=1}^{n-1} \nu_a \xi^j{}_a \overline{\xi^k{}_a} [t_0 (\chi_{\bar{k} j} + \underline{u}_{\bar{k} j}) - \eta \rho_{\bar{k} j}], \nonumber\\
T_2 &=&  \sum_{a=1}^{n-1} (\nu_a \xi^j{}_a \overline{\xi^k{}_a} \ell_{\bar{k}} \rho_j + \nu_a \xi^j{}_a \overline{\xi^k{}_a} \ell_j \rho_{\bar{k}}), \nonumber\\
T_3 &=&\sum_{a=1}^{n-1} \nu_a \xi^j{}_a \overline{\xi^k{}_a} (\ell_i z^i + \ell_{\bar{i}} \bar{z}^i) \rho_{\bar{k} j}, \nonumber\\
T_4 &=& 2M \sum_{a=1}^{n-1}\nu_a \xi^j{}_a \overline{\xi^k{}_a} \partial_j D \partial_{\bar{k}}D + (2M D)\sum_{a=1}^{n-1} \nu_a \xi^j{}_a \overline{\xi^k{}_a} \partial_j \partial_{\bar{k}} D . \nonumber
\eea
We start with the term $T_1$. At the origin, we have
\bea
& \ &   {1 \over t_0} \sum_{a=1}^{n-1} \nu_a \xi^j{}_a \overline{\xi^k{}_a} (t_0 (\chi_{\bar{k} j} + \underline{u}_{\bar{k} j}) - \eta \rho_{\bar{k} j}) (0) \nonumber\\
&=& {1 \over t_0} \sum_{a=1}^{n-1} \nu_a (A_{t_0})_{\bar{a} a} = {1 \over t_0} \sum_{a=1}^{n-1} \nu_a \tilde{\lambda}_a = 0,
\eea
by choice of $\xi_a(0)=\p_a$ and since $A_{t_0}$ is diagonal. Therefore $T_1(0)=0$ and
\[
T_1 = m_i z^i + m_{\bar{i}} \bar{z}^i + O(|z|^2),
\]
where $m_i$ are bounded constants depending only on $(X,\alpha), \chi, \underline{u}$. 
\smallskip
\par The condition that the vector fields $\xi_a$ are tangential to the level sets of $D(z)$ gives the relation
\[
0 =\xi^j{}_a \p_j D = - \xi^j{}_a \rho_j + \tau \xi^j{}_a \bar{z}^j.
\]
Therefore
\[
T_2 = \tau \sum_{a=1}^{n-1} \nu_a \xi^j{}_a \overline{\xi^k{}_a}  (\ell_j z^k + \ell_{\bar{k}} \bar{z}^j).
\]
We expand
\[
T_2 =  \tau \sum_{a=1}^{n-1} \{ \nu_a \xi^j{}_a \overline{\xi^k{}_a}(0) + O(|z|) \} (\ell_j z^k + \ell_{\bar{k}} \bar{z}^j).
\]
Since $\xi_a(0) = \p_a$, we have
\[
T_2 = \tau \sum_{a=1}^{n-1} ( \nu_a \ell_a z^a + \nu_a \ell_{\bar{a}} \bar{z}^a) + O(|z|^2).
\]
Next, we write
\[
T_3 = (\ell_i z^i + \ell_{\bar{i}} \bar{z}^i) (\Lambda_\nu \ddb \rho(0)) + O(|z|^2).
\]
At the origin,
\[
\Lambda_\nu \ddb \rho (0)= \sum_{\alpha=1}^{n-1} \nu_\alpha \rho_{\bar{\alpha} \alpha}(0),
\]
hence
\[
T_3 = (\ell_i z^i + \ell_{\bar{i}} \bar{z}^i) \sum_{\alpha=1}^{n-1} \nu_\alpha \rho_{\bar{\alpha} \alpha}(0)  + O(|z|^2).
\]
Therefore
\bea
T_1 + T_2 + T_3 &=& 2 {\rm Re} \, \sum_{i=1}^{n-1} \{ m_i + \tau \nu_i \ell_i + \ell_i \sum_{\alpha=1}^{n-1} \nu_\alpha \rho_{\bar{\alpha} \alpha}(0)   \} z^i \nonumber\\
&&+ 2 {\rm Re} \, \{ m_n + \ell_n \sum_{\alpha=1}^{n-1} \nu_\alpha \rho_{\bar{\alpha} \alpha }(0) \} z^n + O(|z|^2).
\eea
For $i \in \{1,\dots,n\}$, choose
\[
\ell_i \left( \tau \nu_i + \sum_{\alpha=1}^{n-1} \nu_\alpha \rho_{\bar{\alpha} \alpha }(0) \right) = - m_i,
\]
where here we let $\nu_n=0$. This choice is possible for any $\tau>0$, since by estimate (\ref{a-1}) we have
\[
\tau \nu_i + \sum_{\alpha=1}^{n-1} \nu_\alpha \rho_{\bar{\alpha} \alpha }(0)  \geq a_1.
\]
Therefore, the linear terms cancel exactly, and we have
\be
T_1 + T_2 + T_3 \leq C |z|^2, \ \ |\ell_i| \leq {2 |m_i| \over a_1}.
\ee
Next, since $\xi_a$ is tangential to level sets of $D$,
\[
T_4 = (2MD) \Lambda_\nu \ddb D.
\]
At the origin,
\[
\Lambda_\nu \ddb D(0) = \sum_{\alpha=1}^{n-1} \nu_\alpha (-\rho_{\bar{\alpha} \alpha}(0) + \tau).
\]
Since we have the estimate (\ref{a-1}), we conclude
\[
\Lambda_\nu \ddb D(0) \leq -a_1 + \tau \sum \nu_a \leq - {a_1 \over 2}
\]
for $0<\tau \leq a_1/2$. For $\delta>0$ sufficiently small, we thus have
\[
\Lambda_\nu \ddb D \leq - {a_1 \over 4}
\]
in $\Omega_\delta$. Since $D \geq \tau |z|^2$,
\[
T_4 \leq -{M \tau a_1 \over 2} |z|^2.
\]
It follows that
\[
T_1 + T_2 + T_3 + T_4 \leq C |z|^2 - {M \tau a_1 \over 2} |z|^2 \leq 0,
\]
for $M \geq 2C (\tau a_1)^{-1}$. This completes the proof of Lemma \ref{lemma-Lambda-w}. $\qed$

\bigskip
\par We now return to our goal of establishing $u(z) \leq \Psi(z)$ in $\Omega_\delta$. Let $W = \alpha^{-1}(\chi + \ddb w)$ with eigenvalues $\mu_1 \geq \dots \geq \mu_{n}$. At a point $p \in \Omega_\delta$, we take new coordinates such that $\alpha_{\bar{k} j} = \delta_{kj}$ and $W^i{}_j = \mu_i \delta^i{}_j$. In these coordinates we write $\xi_a= \xi_a^{i}\del_{z^i}$ with $\xi{}^i{}_a$ a unitary matrix, and
\[
\Lambda_\nu (\chi+ \ddb w) = \sum_{a=1}^{n-1} \sum_{i=1}^n |\xi^i{}_a|^2 \nu_a \mu_i.
\]
Let $\nu_0=0$, and let $\xi_0$ be such that $\{ \xi_a \}_{a=0}^{n-1}$ is a local unitary frame for $T^{1,0}X$. Then
\[
\Lambda_\nu (\chi+ \ddb w) = \sum_{a=0}^{n-1} \sum_{i=1}^n |\xi^i{}_a|^2 \nu_a \mu_i,
\]
\[
0 = \nu_0 \leq \nu_1 \leq \dots \leq \nu_{n-1}, \ \ \ \ \mu_1 \geq \dots \geq \mu_n, \ \ \ \ \xi^\dagger \xi = I.
\]
We will follow here the argument of \cite{CNS3}. The matrix $Q^i{}_j = |\xi^i{}_{j-1}|^2$ is a doubly stochastic matrix, and by the Birkhoff-Von Neumann decomposition theorem it can be written as
\[
Q = \sum c_k P_k,
\]
where $c_k \geq 0$ satisfy $\sum_k c_k = 1$, and $P_k$ are permutation matrices. Then
\[
\Lambda_\nu (\chi+ \ddb w) = \sum_k c_k \sum_{i=1}^n \nu_{\beta_k(i-1)} \mu_i,
\]
where the $\beta_k$ are permutations. The minimal configuration is attained by
\[
\sum_{i=1}^n \nu_{\beta_k(i-1)} \mu_i \geq \sum_{i=1}^{n} \nu_{i-1} \mu_{i},
\]
hence Lemma \ref{lemma-Lambda-w} implies
\[
0 \geq \Lambda_\nu (\chi+ \ddb w) \geq \sum_{i=1}^{n} \nu_{i-1} \mu_{i}.
\]
Thus
\[
0 \geq \sum_{a=1}^{n-1} \nu_a \mu_{a+1},
\]
and the vector $(\mu_2, \dots, \mu_{n})$ is outside of $\Gamma_{k-1}\subset \mathbb{R}^{n-1}$ by (\ref{hyperplane}). It follows that $\mu = (\mu_1, \dots , \mu_n)$ is outside of $\Gamma_k\subset \mathbb{R}^n$. Let $\sigma = \inf_X \psi > 0$ and
\[
\Gamma_k{}^\sigma = \{ \lambda \in \Gamma_k : \sigma_k(\lambda) \geq \sigma \}.
\]
The eigenvalues of $W$ all lie in a bounded set of maximal radius $R$ determined by $(X,\alpha)$, $\chi$, $\underline{u}$. There exists an $\epsilon_0>0$ depending on $R$ and $\sigma$ such that $\mu + \epsilon_0 {\bf 1} \notin \Gamma_k{}^\sigma$ at all points in $X$. 
\smallskip
\par We have
\[
\chi_{\bar{k} j} + \Psi_{\bar{k} j} = (\chi_{\bar{k} j} +  w_{\bar{k} j}) + \epsilon \delta_{kj}.
\]
Let $\tilde{\mu} = (\tilde{\mu}_1,\dots,\tilde{\mu}_n)$ denote the eigenvalues of $\tilde{W} = \alpha^{-1} (\chi + \ddb \Psi)$, arranged in decreasing order as usual. By the Weyl inequality,
\[
\tilde{\mu}_i \leq \mu_i + \epsilon \| \alpha^{-1} \|_{\Omega_\delta} \leq \mu_i + \epsilon_0,
\] 
for $\epsilon = \epsilon_0 / \| \alpha^{-1} \|_{\Omega_\delta}$. Therefore 
\be \label{out-Gamma}
\tilde{\mu} \notin \Gamma_k{}^\sigma.
\ee
Next, we adjust our constants such that $u \leq \Psi$ on $\partial \Omega_\delta$. On $\p B_\delta \cap X$, 
\[
w - u \geq (\underline{u}-u) - {\eta \over t_0} \rho(z) + (\ell_i z^i + \ell_{\bar{i}} \bar{z}^i) \rho(z) + M (-\rho(z) + \tau \delta^2)^2 
\]
Since $t_0 \geq 1/2$, $\rho \leq 0$, and $\eta$ is bounded above, we have
\[
w - u \geq -C + M \tau^2 \delta^4.
\]
We previously required $M \geq 2 C (\tau a_1)^{-1}$, and we can increase $M$ to also guarantee $M \geq C(\tau^2 \delta^4)^{-1}$.  It follows that
\[
w - u \geq 0,
\]
on $\p B_\delta \cap X$.  Therefore
\[
\Psi - u \geq \epsilon \delta^2 - {\epsilon \over C_0} x^n \geq 0,
\]
for $C_0 \geq \delta^{-1}$ on $\p B_\delta \cap X$.
\smallskip
\par Next, we consider $B_\delta \cap \p X$. Here we have
\[
w-u = M \tau^2 |z|^4 \geq 0.
\]
Thus
\[
\Psi - u \geq \epsilon |z|^2 - {\epsilon \over C_0} x^n \geq \epsilon |z|^2 - {\epsilon \over C_0} O(|z|^2)
\]
since $x_n= O(|z|^2)$ as a consequence of $\rho = 0$, see e.g. (\ref{rho-expression}). We choose $C_0$ large enough such that $u \leq \Psi$ on $B_\delta \cap \partial X$. Therefore, $u \leq \Psi$ on $\partial \Omega_\delta$. 
\smallskip
\par Putting everything together, we have $u \leq \Psi$ on $\partial \Omega_\delta$ and $\tilde{\mu} \notin \Gamma_k{}^\sigma$ (\ref{out-Gamma}). By the maximum principle Lemma \ref{lem-max-princ}, we conclude $u \leq \Psi$ in $\Omega_\delta$.

\section{Blow-up argument}
In this section, we combine the second order estimate with a blow-up argument to obtain uniform bounds and prove Theorem \ref{thm-c2-uniform}.  We prove

\begin{prop} \label{prop-c1-est}
Let $(X,\alpha)$ be a compact Hermitian manifold with boundary. Let $\chi \in \Gamma_k(X,\alpha)$ be a $(1,1)$ form, $\psi \in C^\infty(X)$ a smooth function satisfying $\psi >0$, and $\phi \in C^\infty(\p X, \mathbb{R})$. Suppose $u \in C^4(\overline{X}, \mathbb{R})$ solves the equation
\[
\sigma_{k}(\lambda) = \psi, \ \ u|_{\p X} = \phi,
\]
where $\lambda \in \Gamma_k$ are the eigenvalues of $\chi + \ddb u$ with respect to $\alpha$. Suppose there exists a subsolution $\underline{u} \in C^\infty(\overline{X},\mathbb{R})$ satisfying
\[
\sigma_k(\underline{\lambda}) \geq \psi, \ \ \underline{u}|_{\p X} = \phi
\]
where $\underline{\lambda} \in \Gamma_k$ are the eigenvalues of $\chi + \ddb \underline{u}$ with respect to $\alpha$. Then
\[
 \| \nabla u \|_{L^\infty(X,\alpha)} \leq C,
\]
where $C$ depends on $(X,\alpha)$, $\underline{u}, \phi$, $|\psi|_{C^{2}(X)}, \inf_{X} \psi$, and $\chi$.
\end{prop}

\begin{proof}
We argue by contradiction. Suppose there exists a sequence of functions $u_i \in C^4(\overline{X},\mathbb{R})$ solving
\[
\sigma_{k}(\lambda(u_i)) = \psi_i, \ \ \lambda \in \Gamma_{k}
\]
with boundary conditions
\[
u_i|_{\p X} = \phi,
\]
with $\psi_i$ uniformly bounded in $C^{2}$ with $\inf_{X} \psi_i \geq a >0$.  We also assume that $\underline{u}$ is satisfies 
\[
\underline{u}|_{\del X} = \phi, \qquad \sigma_{k}(\underline{\lambda}) \geq \psi_i
\]
for all $i$.  To obtain a contradiction we assume there are points $p_i \in X$ with
\[
\| \nabla u_i \|_{L^{\infty}} = |\nabla u_i| (p_i) = M_i,
\]
and $M_i \rightarrow \infty$. After taking a subsequence we may assume that $p_i \rightarrow p_\infty$ for some point $p_\infty \in \overline{X}$.  By Proposition~\ref{prop: c2-est} we have
\be \label{scaled-c2}
\|\ddb u_i \|_{L^{\infty}(X)} \leq C (1+M_i^2)
\ee
for a uniform constant $C$, independent of $i$.
\bigskip
\par Case 1: We begin with the case when $p_\infty$ is in the interior of $X$. Choose a small coordinate ball centered at $p_\infty$ disjoint from the boundary centered at $p_\infty$, and such that $\alpha$ is the identity at $0$. We can assume that all $p_i$ are within this coordinate ball. Let $R>0$, and define $\hat{u}_i : B_R(0) \rightarrow \mathbb{R}$ by
\[
\hat{u}_i(z) = u_i(M_i^{-1} z + p_i),
\]
which is well-defined for all $i$ such that $M_i$ is large enough.  We have
\be \label{blow-up1}
| \nabla \hat{u}_i |(0) = 1,
\ee
and by (\ref{scaled-c2}),
\[
\| \hat{u}_i \|_{L^\infty(B_{R}(0))} + \| \Delta \hat{u}_i \|_{L^\infty(B_{R}(0))} \leq C.
\]
Consequently, standard elliptic theory gives for any $0<\gamma<1$, the estimate
\be \label{blow-up2}
\| \hat{u}_i \|_{C^{1,\gamma}(B_{\frac{R}{2}}(0))} \leq C.
\ee
Fix $\gamma\in (0,1)$.  After taking a subsequence we have $\hat{u}_i \rightarrow u_\infty$ in $C^{1,\frac{\gamma}{2}}(B_\frac{R}{2}(0))$. In fact, by letting $R \rightarrow \infty$, and taking a diagonal subsequence we have a bounded function
\[
u_\infty : \mathbb{C}^n \rightarrow \mathbb{R},
\]
such that $\hat{u}_i \rightarrow u_\infty$ in $C^{1,\frac{\gamma}{2}}$ on compact sets and
\[
|\nabla u_\infty|(0)=1.
\]
If we change coordinates and let
\[
w = M_i(z - p_i),
\]
then the equation satisfied by $\hat{u}_i$ is
\bea
& \ & \left({1 \over M_i^2} \chi_{\bar{k} j}(z) \sqrt{-1} dw^j \wedge d \bar{w}^k + \ddb \hat{u}_i \right)^k \wedge \left( {1 \over M_i^2} \alpha_{\bar{k} j} (z) \sqrt{-1} dw^j \wedge d \bar{w}^k \right)^{n-k} \nonumber\\
&=& \psi_i(z) \left({1 \over M_i^2} \alpha_{\bar{k} j} (z) \sqrt{-1} dw^j \wedge d \bar{w}^k \right)^{n}. \nonumber
\eea
Since the $\hat{u}_i$ converge locally uniformly, we can take a limit (e.g. \cite{Dem} Chapter III, Cor. 3.6, \cite{DK17}) of the equation and obtain
\[
(\ddb u_\infty)^{k} \wedge \beta^{n-k} = 0,
\]
in the Bedford-Taylor sense \cite{BT}, or B\l ocki \cite{Blocki05}, where
\[
\beta = \sum_{p=1}^n \sqrt{-1} dw^p \wedge d \bar{w}^p.
\]
The fact that $u_\infty$ is bounded and non-constant violates the Liouville theorem of Dinew-Ko\l odziej \cite{DK17}, giving a contradiction.
\bigskip
\par Case 2: We now address the case when $p_\infty \in \partial X$. Let $\Omega \subset X$ be a coordinate chart centered at $p_\infty$, making $\alpha_{\bar{k}j}(p_{\infty})= \delta_{\bar{k}j}$.  That is, $\Omega$ is identified with the subset $\{z\in B_{2s} : \rho(z) \leq 0\} \subset \mathbb{C}^n$, where $B_{2s}\subset \mathbb{C}^n$ is the euclidean ball of radius $2s$ centered at $0=p_{\infty}$, and $\rho: B_{2s} \rightarrow \mathbb{R}$ is a boundary defining submersion (ie. $\del X \cap \Omega = \{ \rho =0\}$, and $X\cap \Omega \subset \{ \rho \leq 0\}$).   By going far enough along in the sequence, we may assume that all $p_i$ are in $\Omega$, and $|p_i|<s$ in local coordinates.  For $i$ sufficiently large, there is a unique point $y_i \in \del X \cap \Omega $ minimizing the euclidean distance between $p_i$ and $\del X$.  Let $r_i= |p_i-y_i|$, and note that $r_i \rightarrow 0$ as $i\rightarrow \infty$.  By the triangle inequality we have $y_i\rightarrow p_{\infty}$.  As before, we define
\be
\hat{u}_i(z) = u_i (M_i^{-1} z + p_i). 
\ee
We have
\be \label{blow-up3}
| \nabla \hat{u}_i |(0) = 1,
\ee
and by (\ref{scaled-c2}),
\[
\| \hat{u}_i \|_{L^\infty(\Omega_i)} + \| \Delta \hat{u}_i \|_{L^\infty(\Omega_{i})} \leq C.
\]
where $\Omega_i$ is the set of $z$ such that $M_i^{-1} z + p_i \in B_{2s} \cap \{\rho \leq 0\}$.  Let $\hat{\rho}_i(z) = \rho(M_i^{-1}z+p_i)$.  Since $|p_i|<s$ we have that $\Omega_i \supset B_{sM_i} \cap \{\hat{\rho}_i \leq 0\}$.  Consequently, standard elliptic theory (see, for example \cite[Chapter 8]{GT}) gives for any $0<\gamma<1$, the estimate
\be \label{blow-up4}
\| \hat{u}_i \|_{C^{1,\gamma}(B_{\frac{sM_i}{2}}\cap \{\hat{\rho}_{i} \leq 0\})} \leq C.
\ee
for a uniform constant $C$ independent of $i$.
Fix $\gamma \in(0,1)$.  There are now two cases, depending on $\liminf _{i\rightarrow \infty}M_ir_i$.
\bigskip
\par Case 2a:  Assume that $\liminf _{i\rightarrow \infty}M_i r_i = +\infty$.  Then, after passing to a subsequence, $\hat{u}_i(z)$ is defined on the ball $B_{\frac{M_ir_i}{2}} \subset B_{\frac{sM_i}{2}}\cap \{\hat{\rho}_{i} \leq 0\}$, and $M_ir_i\rightarrow \infty$.  Therefore, taking further subsequences we have that  $\hat{u}_i$ converges in $C^{1,\frac{\gamma}{2}}$ on compact sets to $u_\infty \in C^{1,\frac{\gamma}{2}}(\mathbb{C}^n)$.  As in Case 1, this contradicts the Liouville theorem of Dinew-Ko\l odziej \cite{DK17}.
\bigskip
\par Case 2b:  Assume that $\liminf_{i\rightarrow \infty} M_ir_i = L \in [0, +\infty)$.  Up to taking a subsequence, we can assume that $\lim_{i\rightarrow \infty} M_ir_i = L$.  Choose a ball of radius $r^*$ centered at a point $y^*$ such that $\overline{B_{r^*}(y^*)} \subset B_{s}\cap \{\rho \leq 0\}$ and $B_{r^*}(y^*)$ is tangent to $\{\rho=0\}$ at $p_{\infty}$.  After possibly shrinking $r^{*}$ slightly (and moving $y^*$ accordingly) we can assume that, for $i$ sufficiently large there are points $y_i^{*}$ such that $\overline{B_{r^*}(y_i^*)} \subset B_{s}\cap \{\rho \leq 0\}$ and  $B_{r^*}(y_i^*)$ is tangent to $\{\rho=0\}$ at $y_i$.  Furthermore, $y_i^*\rightarrow y^*$ as $i\rightarrow \infty$ since $y_i \rightarrow p_{\infty}$.  For convenience, redefine $\hat{u}_i$ so that $y_i$ is the origin, and $p_i= (0,\ldots,0, \sqrt{-1}M_ir_i)$.  That is, define
\[
\hat{u}_i = u_i(M_i^{-1}z+ y_i).
\]
and perform a unitary transformation to achieve $p_i= (0,\ldots,0, \sqrt{-1}M_ir_i)$.  Now, since $\overline{B_{r^*}(y_i^*)}\subset B_{s}\cap \{\rho \leq 0\}$ arguing as in case 2a, we obtain uniform $C^{1,\gamma}$ bounds on closed balls $\overline{\hat{B}_{i}}$ of radius $M_ir^*$ contained in the closed upper half-plane $\{{\rm Im}\, z_n\geq  0\}$, and tangent to $\{{\rm Im}\, z_n=0\}$ at the origin.  Furthermore, we have
\[
\hat{u}_{i}(0) = \phi(y_i), \qquad |\nabla \hat{u}_i|(0,\ldots, 0, \sqrt{-1}M_ir_i) =1.
\]
Since the balls $\hat{B}_i$ exhaust the upper half plane $\{{\rm Im}\, z_n > 0\}$, after taking a subsequence we can assume that $\hat{u}_{i}$ converges uniformly in $C^{1,\frac{\gamma}{2}}$ on compact sets of $\{{\rm Im},z_n >0\} \cup \{0\}$ to a limit $u_{\infty}: \{{\rm Im}\, z_n > 0\} \cup \{0\} \rightarrow \mathbb{R}$.  Furthermore, since $M_ir_i \rightarrow L$, and the convergence is in $C^{1,\frac{\gamma}{2}}$ we have
\begin{equation}\label{eq: bdry-blow-up-grad}
|\nabla u_{\infty}|(0,\ldots,0, \sqrt{-1}L)=1.
\end{equation}
\smallskip

\par As before we let $b$ denote the function solving
\[
\alpha^{n-1}\wedge (\chi + \ddb b)=0 \qquad b|_{\del X}= \phi.
\]
By the comparison principle we have
\[
\underline{u} \leq u \leq b.
\]
Let $\hat{\underline{u}}_i = \underline{u}(M_i^{-1}z_i + y_i)$, and $\hat{b}_i = b(M_i^{-1}z + y_i)$, so that
\begin{equation}\label{eq: blow-up-upper-lower-bd}
\hat{\underline{u}}_i \leq \hat{u}_{i} \leq \hat{b}_i, \qquad \hat{\underline{u}}_i(0)= \hat{u}_{i} (0)= \hat{b}_i(0) = \phi(y_i)
\end{equation}

 One easily checks that the sequences $\hat{\underline{u}}_i$ and $\hat{b}_i$ converge in $C^{1,\frac{\gamma}{2}}$ on compact sets of $\{{\rm Im},z_n >0\} \cup \{0\}$ to constant functions $\underline{u}_{\infty} = \phi(p_{\infty}) = b_{\infty}$.  Thanks to~\eqref{eq: blow-up-upper-lower-bd} we conclude that $u_{\infty} = \phi(p_{\infty})$ is a constant, but this contradicts~\eqref{eq: bdry-blow-up-grad}.

\end{proof}

Finally, Theorem~\ref{thm-c2-uniform} and hence also Theorem~\ref{thm-main} follow by combining Proposition~\ref{prop: c2-est} and Proposition~\ref{prop-c1-est}.

\end{document}